\documentclass[11pt]{amsart}
\usepackage{amsmath}
\usepackage{amssymb}
\usepackage{color}

\theoremstyle{plain}
\newtheorem{theorem}{Theorem}[section]
\newtheorem{lemma}[theorem]{Lemma}

\theoremstyle{definition}
\newtheorem{remark}[theorem]{Remark}
\newtheorem{definition}[theorem]{Definition}

\numberwithin{equation}{section}

\def\be{\begin{equation}}
\def\ee{\end{equation}}

\begin{document}

\title[Boundary expansion in domains with conic singularities]
{Boundary Expansion for the Loewner-Nirenberg Problem in domains with conic singularities}


\author{Xumin Jiang}
\address{Department of Mathematics\\
Fordham University\\
The Bronx, NY 10458}  \email{xjiang77@fordham.edu}

\begin{abstract}
We study  asymptotic behaviors of solutions to the 
Loewner-Nirenberg problem in domains with conic singularities and establish asymptotic expansions with respect to two normal directions simultaneously.
The spherical domains over which cones are formed are allowed to have singularities.
An elliptic operator on such spherical domains with coefficients singular on boundary play an important role. Key step is the study of the
eigenvalues growth and 
eigenfunctions 
 estimates.
\end{abstract}

\maketitle

\section{Introduction}\label{sec-Intro}

Let $\Omega$ be a bounded domain in $\mathbb R^n$, for some $n\ge 3$. Consider 
\begin{align}
\label{eq-LN-MainEq} \Delta u  &= \frac14n(n-2) u^{\frac{n+2}{n-2}} \quad\text{in }\Omega,\\
\label{eq-LN-MainCond}u&=\infty\quad\text{on }\partial \Omega.
\end{align}
This is the so-called Loewner-Nirenberg problem, also known as the singular Yamabe problem. For a large class of domains $\Omega$, \eqref{eq-LN-MainEq} and \eqref{eq-LN-MainCond} admit a unique positive solution
$u \in C^\infty (\Omega
)$ . Geometrically, $u^\frac{4}{n-2}\sum_{i=1}^n dx^i \otimes dx^i$ is a complete metric with the constant scalar curvature $-n(n-1)$ on $\Omega$.

 In a pioneering work, Loewner and Nirenberg \cite{Loewner&Nirenberg1974} proved that 
\eqref{eq-LN-MainEq} and \eqref{eq-LN-MainCond} admit a unique positive solution $u\in C^\infty(\Omega)$ and studied asymptotic behaviors of  $u$ under the condtion that $\Omega$ is a $C^2$-domain.
Specifically, they proved that for $d(x)$, the distance function in $\Omega$ to $\partial \Omega$, sufficiently small,
\begin{equation}\label{eq-LN-basic}\big|d^{\frac{n-2}2}(x)u(x)-1\big|\le Cd(x),\end{equation}
where $C$ is a positive constant 
depending only on $n$ and the $C^{1,1}$-norm of $\partial\Omega$.  Kichenassamy \cite{Kichenassamy} expanded further if $\Omega$ has a $C^{2,\alpha}$-boundary.
 Mazzeo \cite{Mazzeo1991} and 
Andersson, Chru\'sciel, and Friedrich \cite{ACF1982CMP} proved that solution $u$ of 
\eqref{eq-LN-MainEq}-\eqref{eq-LN-MainCond} is polyhomogeneous if $\Omega$ has a smooth boundary. If $\Omega$ is a Lipschitz domain, Han and Shen \cite{HanShen1} studied asymptotic behaviors of solutions near singular points on $\partial \Omega$, and proved an estimate similar as \eqref{eq-LN-basic}, under appropriate conditions of the domain near singular points.

In this paper, we study a more basic question and investigate asymptotic behaviors of solutions of \eqref{eq-LN-MainEq}-\eqref{eq-LN-MainCond} along two normal directions if $\Omega$ is a finite cone.
Let $T = \mathbb{R}^{n-k} \times
 T_k$, for some $2\leq k\leq n$,  where $T_{k}\subseteq \mathbb{R}^k$ is the Euclidean cone over some smooth spherical domain $S \subsetneq  \mathbb{S}^{k-1}$. We assume that $\Omega \cap B_1 =  T \cap B_1.$

First we have the following spectral theorem and eigenvalue growth estimate for certain singular elliptic operators. The spectral theorem can be derived by classical theory, as shown  in Appendix \ref{sec-Spec}.
\begin{theorem}\label{thm-Eigen-Growth}
Let $(S, g)$ be a Lipschitz $l$ dimensional  Riemannian manifold with codimension one boundary. Assume that $d$ is a Lipschitz defining function of $\partial S$. Then in the space $H^1_0$ with norm 
\begin{align*}
||u||_d := \left\lbrace \int_S \left(|\nabla u|^2 +\frac{u^2(x)}{d^2(x)}\right) dvol \right\rbrace^\frac{1}{2},
\end{align*}
 the operator
\begin{align*}
L[u] := -\Delta_S u+\frac{\kappa}{d^2(x)}u,
\end{align*}
where $\kappa>0$ is a constant, has a complete set of $L^2(S)$-orthonormal eigenfunctions $\{\phi_j\}_{j=1}^\infty$. In addition, if the interior of $S$  is smooth and for any integer $m\in [1, \frac{l}{2}+1]$,
\begin{align}\label{eq-Eigen-Cond}
 |\nabla_S^m d|\leq C(m)d^{1-m}
\end{align}
in $S$,
then the eigenvalues $\lambda_1\leq  \lambda_2\leq \cdots \leq \lambda_i\leq \cdots$ satisfy
\begin{align*}
\lambda_i >C i^\frac{2}{l}
\end{align*}
and the eigenfunctions $\phi_i$'s satisfy
\begin{align*}
|\phi_i(x)| \leq C i^{\frac{l}{4}},
\end{align*}
where $C$ is a positive constant depending only on $l$ and $S$, independent of $i$.
\end{theorem}
We remark that $\lambda_i < Ci^\frac{2}{l}$ can also be derived through a similar argument as in Li and Yau \cite{LiYau}.


According to  \cite{HanShen1}, there is a unique solution $u_T$ to \eqref{eq-LN-MainEq}-\eqref{eq-LN-MainCond} in $\Omega=T$. In addition, $u_T = r^{-\frac{n-2}{2}} u_S$, where $r$ is the radial coordinate of $T_k$, and $u_S$ is defined on $S$, satisfying
\begin{align}\begin{split}\label{eq-uS}
\Delta_{\mathbb{S}^{k-1}}u_S -
\left(\frac{n-2}{2} \right)\left(k-1-\frac{n}{2} \right)  u_S &=\frac{1}{4}n(n-2)  u_S^{\frac{n+2}{n-2}} \text{ in }S,\\
\qquad  u_S &= \infty \text{ on }\partial S.
\end{split}
\end{align}

In this paper, we assume that $T$ is Lipschitz.
For a fixed number $M>0$, denote
\begin{align}\label{eq-TM}
{T}_M= \{(x^\prime_{n-k}, r, \theta)\in T: |x^\prime_{n-k}|< M, 0<r<M, \theta \in S\}.
\end{align} 


Let $(d_S, z_S)$ be the geodesic coordinates of $S$ near $\partial S$, where $d_S = \text{dist}(\cdot , \partial S)$ (modified smooth at points far away from $\partial S$) and $z_S$ denotes the coordinates on $\partial S$.
We are ready to define the boundary expansion (polyhomogeneity) with respect to the $d_S$ coordinate.
\begin{definition}\label{def-exp-ds} For any $b\in \mathbb{N}$,
we say a function $w$ has a boundary expansion of order $d_S^b$ in $T_M$,
 if there are smooth $c_{i, j}$'s  and $R_b$ defined in $T_M$, such that for any fixed $r\in (0, M)$,
 \begin{align}\label{eq-ExpOrdI2}
w= \sum_{l=0}^b  \sum_{m=0}^{N_l} c_{l, m} d_S^l (\log d_S)^m +R_b,
\end{align}
 where $c_{l,m}$'s are independent of $d_S$, and for any $p, q, i, j\in \mathbb{N}, \alpha\in (0,1)$,
\begin{align}
|r^pD_r^p D_{x^\prime_{n-k}}^q D_{z_S}^i c_{l, m} | &\leq C(T, M, l, m, p, q, i),\label{eq-Cij-1}\\
|r^p  D_r^p D_{x^\prime_{n-k}}^q D_{z_S}^i D_{d_S}^j R_{b} | &\leq C(T, M, b, p, q, i, j,  \alpha) d_S^{b+\alpha-j},\label{eq-Rb}
\end{align}
where $z_S$ denotes the coordinates on $\partial S$.  If $d_S$ is far away from $0$, \eqref{eq-Rb} denotes that  $r^pD_r^p D_{x^\prime_{n-k}}^q  R_{b}$ and its covariant derivatives in $S$ are bounded. 
\end{definition}

In this paper, $N_l= \lfloor \frac{l}{n} \rfloor.$ 

Next we define the boundary expansion with respect to the $r$ direction with $O(d_S^\tau)$ coefficients for some $\tau \in \mathbb{R}$.

\begin{definition}\label{def-exp-r}
Given an index set $J \subseteq \mathbb{R}^+$,
We say a function $v$ has a boundary expansion of order $r^a (a \in J)$ with  $O(d_S^\tau)$ coefficients in $T_M$, if there are functions
$\tilde{c}_{i, j}$'s, $\tilde{R}_a$ defined in $T_M$, and an $\epsilon>0$, such that,
\begin{align}\label{eq-Exp-r}
v= \sum_{i\in J, i\leq a}\sum_{j=0}^{\tilde{N}_i} \tilde{c}_{i,j} r^i(\log r)^j +\tilde{R}_a
\end{align}
where $d_S^{-\tau} \tilde{ c}_{i,j}$ and $d_S^{-\tau}  \cdot r^{-a-\epsilon}  \tilde{R}_a$ have boundary expansions up to order $d_S^b$ for any integer $b\in \mathbb{N}$. In addition, $\tilde{c}_{i, j}$'s are independent of $r$. 
\end{definition}
We remark that we can switch the order of $r$ and $d_S$, and  define the boundary expansion of order $d_S^\tau$ with  $O(r^a)$ coefficients similarly. However it is equivalent to Definition \ref{def-exp-r}. $r$ and $d_S$ have the same status in the expansion.

In this paper, $\tau$ could be $\frac{n+2}{2}$ or $\frac{n-2}{2}+n$. Notice that in \eqref{eq-Exp-r}, when $d_S$ is far away from $0$,  $\tilde{c}_{i,j}, \tilde{R}_a$ are smooth in $S$ and the corresponding norms are bounded, in which case, \eqref{eq-Exp-r} is simply an  expansion in the single $r$ direction.

The next is the main theorem of this paper.
\begin{theorem}\label{thm-Main} For an $M>0$,  $n\geq 3$, assume that $u\in C^2(T_M)$ is a positive solution to 
\begin{align}
\label{eq-LN-MainEq1} \Delta u  &= \frac14n(n-2) u^{\frac{n+2}{n-2}} \quad\text{in } T_M,\\
\label{eq-LN-MainCond1}u&=\infty\quad\text{on }\partial T \cap \partial T_M,
\end{align}
 where $T_M$ is defined as \eqref{eq-TM}. Then there is a countable index set $J \subseteq \mathbb{R}^+$, such that for any $a\in J$, $u- u_T$ has an expansion of order $r^a$ with $O(d_S^\frac{n+2}{2})$ coefficients  in $T_\frac{M}{2}$.
\end{theorem}

A corollary of
Theorem \ref{thm-Main}  is that
\begin{align*}
u-u_T= O(r^{j_1}d_S^\frac{n+2}{2})
\end{align*}
where $j_1$ is the smallest element in $J$, which can be computed from  the first eigenvalue of $L_S$.

We now briefly describe  the proof of Theorem \ref{thm-Main}. 
Denote $v= u-u_T$. Then $v$ satisfies the equation
\begin{align}
r^2\Delta_{\mathbb{R}^{n-k}}v+ Nv +L_{S}  v=r^\frac{n-2}{2} u_S^{\frac{6-n}{n-2}} v^2 \cdot F (u_T^{-1}v), \label{eq-polar}
\end{align}
where $F$ is an analytic function, and
\begin{align}
Nv &=  r^2 v_{rr}+(k-1)rv_r,\nonumber\\
L_S v &= \Delta_{S} v - \frac{n(n+2)}{4}  u_S^{\frac{4}{n-2}} v.\label{eq-LS}
\end{align}
In $T_\frac{M}{2}$, we first prove that $v$ is bounded. Then around any fixed $P\in T_\frac{M}{2}$,  we apply the rescaling $t=r/r(P)$, under which the main equation \eqref{eq-polar} is only singular in the $d_S$ direction. So we can apply teniques in  \cite{Mazzeo1991}, \cite{ACF1982CMP} or \cite{HanJiang} to derive that $d_S^{-\frac{n+2}{2}} v$ has a boundary expansion of order $d_S^b$ for any $b\in \mathbb{N}$, in the sense of Definition \ref{def-exp-ds}. In $T_\frac{M}{2}$, for any fixed $x^\prime_{n-k}, r$,  the spectral theorem implies that,
\begin{align}\label{eq-v-phi}
v= \sum_i^\infty A_i(x_{n-k}^\prime, r)\phi_i(\theta),
\end{align}
where the Fourier coefficients $A_i$'s satisfy
 ODE's of form
\begin{align*}
r^2 A_i^{\prime\prime}+(k-1) rA_i^\prime -\lambda_i A_i=  \tilde{F}_i.
\end{align*}
Solving the ODE, and plugging into \eqref{eq-v-phi}, we derive an expansion of $v$ of form \eqref{eq-Exp-r}, where the coefficients $\tilde{c}_{i,j}$ and remainder $\tilde{R}_a$ satisfy  equations of form
\begin{align}\label{eq-l1l2}
 r^2 w_{rr} + l_1 rw_r+ l_2 w+ L_S w= \tilde{F},
\end{align}
for some $l_1, l_2\in \mathbb{R}$. This is sufficient to show that they have expansions in the $d_S$ direction in the sense of Definition \ref{def-exp-ds}, and derive  Theorem \ref{thm-Main}.

In Section \ref{sec-Liou}, we take the Liouville's equation as an example to interpret Theorem \ref{thm-Main}. In Section \ref{sec-comp}, we show that $u-u_T$ is bounded. In Section \ref{sec-exp-ds}, we show the expansion of $u-u_T$ in the $d_S$ direction. In Section \ref{sec-EigenG}, we prove Theorem \ref{thm-Eigen-Growth}. In Section \ref{sec-SmoothS}, we prove Theorem \ref{thm-Main}.

Thanks to Zheng-Chao Han and Yalong Shi for helpful discussions.

\bigskip
\section{An example: Liouville's equation}\label{sec-Liou}

In this section, we study the Liouville's equation with conic singularities.
\subsection{Smooth case}
Consider the following problem
\begin{align}\begin{split} \label{eq-LiouEq}
\Delta  u&= e^{2u} \text{  in  } \Omega,\\
u&=\infty \text{  on  } \partial \Omega.
\end{split}
\end{align}
Geometrically, $e^{2u}(dx^1\otimes dx^1+dx^2\otimes dx^2)$ is a complete metric with constant Gauss curvature $-1$ on $\Omega.$
By maximum principle,
\begin{theorem}
\label{Thm-LocTanAnaly}
Let $\Omega$ be a bounded domain in $\mathbb{R}^2$ and $\partial\Omega$ be a $C^{1,\alpha}$ near $x_0\in \partial \Omega$ for some $\alpha \in (0,1]$. Suppose $u \in C^{\infty}(\Omega)$ is a solution of \eqref{eq-LiouEq}. Then,
\begin{align}\label{eq-TanEsmNoN}
|u+\log d|\leq d^\alpha \text{\,\,in  } \Omega\cap B_r(x_0),
\end{align} 
where $d$ is the distance to $\partial \Omega$, and $r$ and $C$ are possible constants depending only on $\alpha$ and the geometry of $\Omega.$
\end{theorem}

Denote $v=u+\log d$. Then $v$ satisfies the equation,
\begin{align}\begin{split} \label{eq-MainV}
\Delta  v-\frac{2v}{d^2}= d^{-2} (e^{2v}-1-2v)+\Delta d \frac{1}{d}.
\end{split}
\end{align}

By \cite{HanJiang}, if locally $\partial \Omega$ is smooth, then
\begin{align*}
u=-\log d +\frac{1}{2}\kappa d +\sum_{i=2}^k c_{i} d^i+R_k,
\end{align*}
where $\kappa$ is the curvature of the boundary curve. 
The remainder $R_k$ is $C^{k,\epsilon}$ for any $\epsilon \in (0, \alpha)$, and
\begin{align*}
R_k=O(d^{k+\alpha}).
\end{align*}

If a portion of $\partial \Omega$ is a straight segment, we know that by \cite{HanJiang1}, around this boundary segment, $u+\log d$ is analytic in $d$ and the boundary segment coordinate, i.e.,
\begin{align}\label{eq-u-Logd-K-0}
u\equiv-\log d +c_2d^2+c_3d^3 +\cdots,
\end{align}
where the coefficients $c_2, c_3, \cdots$ are analytic functions on the boundary segment.

\bigskip
\subsection{Singular planer case}

First consider the case when locally $\bar{\Omega}$ coincides with the first quadrant in $\mathbb{R}^2$. Under the conformal transformation $\tilde{z}=z^2$, the local boundary portion is mapped to a straight segment, while the Liouville's equation is kept invariant. So we know by \eqref{eq-u-Logd-K-0}, the solution is
\begin{align*}
u\equiv\frac{1}{2}\log (\frac{1}{x^2}+\frac{1}{y^2})+c_2 \tilde{y}^2+c_3\tilde{y}^3+\cdots+ c_i\tilde{y}^i+\cdots,
\end{align*} 
which is an analytic convergent series, where
\begin{align*}
\tilde{y}=2xy,
\end{align*}
and $c_i$'s are analytic in $\tilde{x}=x^2-y^2$. We can also express the solution as
\begin{align*}
u=\frac{1}{2}\log (\frac{1}{x^2}+\frac{1}{y^2})+x^2 y^2 \overline{R},
\end{align*}
where $\overline{R}$ is analytic in $x,y$.

If locally $\bar{\Omega}$ is a sector with angle $\mu \pi$ centered at the origin, i.e. $\bar{\Omega}=\{0\leq \theta \leq \mu \pi\}$ under the polar coordinates, then by the conformal transformation $\tilde{z}=z^{\frac{1}{\mu}}$, we know the solution
\begin{align*}
u\equiv \frac{1}{2} \log \left( \frac{1}{4\mu^2} \cdot \left( \frac{1}{\tilde{x}^2} + \frac{1}{\tilde{y}^2} \right) \right)+ c_2\tilde{y}^2 +\cdots+ c_i \tilde{y}^i +\cdots,
\end{align*}
where $\tilde{y}=r^\frac{1}{\mu} \sin \frac{\theta}{\mu}$ and $c_i$'s are analytic in $\tilde{x}=r^{\frac{1}{\mu}}\cos \frac{\theta}{\mu}$. 

\begin{remark} We can use the coordinates $\bar{x}= r^{\frac{1}{2\mu}}\cos \frac{\theta}{2\mu}, \bar{y}= r^{\frac{1}{2\mu}}\sin \frac{\theta}{2\mu}$, then
\begin{align*}
u+\log(\mu r \sin (\frac{\theta}{\mu}))=O(\bar{x}^2\bar{y}^2)
\end{align*}
and is analytic in $\bar{x}, \bar{y}$. 
\end{remark}

However, this method does not work for the equation \eqref{eq-LN-MainEq}-\eqref{eq-LN-MainCond} as in general we do not have a conformal transformation that maps an Euclidean cone to the upper half plane in $\mathbb{R}^n$ ($n\geq 3$).

\bigskip
\subsection{Singular planer case under polar coordinates}
We study the Liouville's equtaion under the polar coordinates. Consider the case that near a boundary point, say the origin, $\Omega$ coincides with the first quadrant in $\mathbb{R}^2$ near the origin. We check that $u_T = \log (\frac{1}{2}r\sin (2\theta))= \log (\frac{1}{x^2}+\frac{1}{y^2})$ is a solution in  $T=\{x>0, y>0\}$.

Assume that $u$ is a solution to  Liouville's equation in $T_M = T\bigcap B(O, M)$, for some constant $M>0$. Denote $v=u - u_T$. Then $v=O(dist(\cdot , \partial T))$ by \cite{HanShen}, and satisfies 
\begin{align*}
v_{xx}+ v_{yy}- \frac{2}{x^2}v-\frac{2}{y^2}v= (\frac{1}{x^2}+\frac{1}{y^2})v^2F(v),
\end{align*}
 in a domain $T_M $, where 
\begin{align*}
F(v)= \frac{e^v-1-v}{v^2}= \frac{1}{2!}+ \frac{v}{3!}+\frac{v^2}{4!}+\cdots
\end{align*} 
  is analytic in $v$. 
Under the polar coordinates, we have
\begin{align}\label{eq-RTheta}
r^2 v_{rr}+r v_r + v_{\theta \theta} -\frac{8}{\sin^2 (2\theta)} v= \frac{4 v^2}{\sin^2 2 \theta} F(v).
\end{align}
For any fixed $r\in (0, M)$, $v=O(\sin^2 (2\theta))$ and is smooth on $[0, \frac{\pi}{2}]$. In fact, by maximum principle, we can show that
\begin{align*}
|v| \leq C \sin^2(2\theta),
\end{align*}
where $C$ is independent of $r$.

Denote $L_\theta= D_{\theta\theta}-\frac{8}{\sin^2 2\theta}$. We consider 
\begin{align*}
X=\{w\in H^1[0, \frac{\pi}{2}]:\int_0^\frac{\pi}{2}\frac{w^2}{\sin^2 (2\theta)} d\theta < \infty\}.
\end{align*}
By appendix \ref{sec-Spec}, and through computation, we see that on $X$, $- L_\theta$ has eigenvalues $\lambda_l= 4l^2$ ($l\geq 2$), and  the corresponding eigenfunctions $\phi_l$'s of form, 
\begin{eqnarray*}\phi_{l}(\theta)=
\begin{cases}
\sum_{j=2}^{l} A_{l, j} \sin^{j} (2\theta), & \text{for }  l \text{ even}\cr \sum_{j=2}^{l-1} B_{l,j} \sin^{j} (2\theta) \cos(2\theta), & \text{for } l \text{ odd}\end{cases}
\end{eqnarray*}
where the coefficients $A_{l,j}, B_{l,j}$ can be derived through formal computation. The first two eigenfunctions are \begin{align*}
\phi_{2}=\sin^{2} (2\theta),\quad \phi_{3}=\sin^{2} (2\theta)\cos(2\theta).
\end{align*} All of the eigenfunctions are  $O(\sin^2(2\theta))$ as $\theta\rightarrow 0^+$ or $\frac{\pi}{2}^-$.

Applying the techniques in the following sections, we can derive a theorem similar to Theorem \ref{thm-Main}.
\begin{theorem} For an $M>0$, assume that $u \in C^2(T\cap B(O, M))$ is a solution to 
\begin{align*}
 \Delta u  &= e^u \quad\text{in } T\cap B(O, M),\\
u&=\infty\quad\text{on }\partial T \cap B(O, M),
\end{align*}
 where $T$ is the first quadrant in $\mathbb{R}^2$. Then there is a  set $\{N_l\}\subseteq \mathbb{N}$, such that for any $b \in \mathbb{N}$, $u+ \frac{1}{2} \log (\frac{1}{x^2}+\frac{1}{y^2})$ has the boundary expansion,  for $r\in (0, \frac{M}{2})$,
 \begin{align}\label{eq-UExp-dim2}
u+\frac{1}{2}\log(\frac{1}{x^2}+\frac{1}{y^2}) =c_4(\theta) r^4+ \sum_{l=3}^{b}\sum_{j=0}^{N_l} c_{2l,j}(\theta)r^{2l} (\log r)^j+ R_{2b}(r,\theta),
 \end{align}
 where $c_4$ and $c_{i,j}$'s are smooth for $\theta\in [0,\frac{\pi}{2}]$, all of which are
$
O(\sin^2(2\theta)),
$ as $\theta \rightarrow 0$ or $\frac{\pi}{2}$. In addition, for any $\alpha \in (0,1)$, $p, j \in \mathbb{N}$, 
 \begin{align*}
|r^p  D_r^p  D_{\theta}^j \left( r^{-2b-\alpha}  (\sin 2\theta)^{-2} R_{2b}\right)| &\leq C(T, M, b, p,  j,  \alpha),
\end{align*}
in $T_\frac{M}{2}$.
\end{theorem}

Here the powers of $r$ are even numbers, as the eigenvalues are explicit and special. There are logarithmic terms on the right hand side of \eqref{eq-UExp-dim2}, as we worked with a general smooth function $F(v)$ in \eqref{eq-RTheta}. We can discuss about how to get ride of the logarithmic terms in \eqref{eq-UExp-dim2} using formal computation, but we skip it as it's not the main concern of this paper.




\bigskip
\section{First comparison with solutions in infinite cones}\label{sec-comp}
 Let $T, S, T_M$ be defined as in the introduction.
Acording to \cite{HanShen1}, if $T$ is Lipschitz,  there exists a unique solution $u_T$ to the Loewner-Nirenberg problem on $T$, such that $u_T(x^\prime_{n-k} ,r, \theta)= r^{-\frac{n-2}{2}}u_S(\theta)$, where $u_S$ satisfies
\eqref{eq-uS}.

Assume that we have a local $C^2$ solution $u$ to \eqref{eq-LN-MainEq}.
 Denote $v= u- u_T$. Then $v$ satisfies
\begin{align}
\Delta v - \frac{n(n+2)}{4} u_T^{\frac{4}{n-2}} v= u_T^{\frac{6-n}{n-2}} v^2 \cdot F (u_T^{-1}v), \label{eq-Euc}
\end{align}
or
\begin{align}
r^2\Delta_{\mathbb{R}^{n-k}}v+ Nv +L_{S}  v=r^\frac{n-2}{2} u_S^{\frac{6-n}{n-2}} v^2 \cdot F (u_T^{-1}v), 
\end{align}
where
\begin{align}
Nv &:=  r^2 v_{rr}+(k-1)rv_r,\nonumber\\
L_S v &:= \Delta_{S} v - \frac{n(n+2)}{4}  u_S^{\frac{4}{n-2}} v,
\end{align}
and $F$ is an analytic function, well defined if $|u_T^{-1}v|<1$.

The next theorem is essentially by  \cite{HanShen1}.
\begin{theorem}\label{thm-CSTBD}
Denote $T_M$ as \eqref{eq-TM}, and assume that $T$ is Lipschitz. Let $u_1, u_2$ be two positive $C^2$ solutions of
\begin{align}\label{eq-LN}
 \Delta u  &= \frac14n(n-2) u^{\frac{n+2}{n-2}},
\end{align}
in $T_M$ for some $M>0$, and $u_1=u_2=+\infty$ on $\partial T \cap \{r<M\}$. Then
\begin{align*} 
|u_1-u_2|\leq C,
\end{align*}
in $T_\frac{M}{2}$, where the constant $C=(\frac{4}{M})^\frac{n-2}{2}$.
\end{theorem}
\begin{proof}
Denote $u_M$ the solution to the Loewner-Nirenberg problem in the Euclidean ball $B(O, M)\subseteq \mathbb{R}^n$ with radius $M$. It's well known that
\begin{align*}
u_M= \left(\frac{2M}{M^2-|x|^2}\right)^\frac{n-2}{2}
\end{align*}
is the explicit solution to \eqref{eq-LN-MainEq} and \eqref{eq-LN-MainCond} in $\Omega=B(O, M).$

Since $T$ is Lipschitz, there is an $e \in T$, such that for any $\epsilon>0$, the translation $T+\epsilon e \subset T$.  For any $0<\epsilon<R$,
\begin{align*}
u_{2, \epsilon}(x) := u_2(x-\epsilon e),
\end{align*}
which is a solution to \eqref{eq-LN} in $(T +\epsilon e) \bigcap B(O, M)$ such that $u_{2, \epsilon}=+\infty$ on $\partial (T +\epsilon e) \cap B_R(0)$. Now that $u_M+ u_{2,\epsilon}$ is a supersolution. So by maximum principle, we have
\begin{align*}
u_1< u_M+ u_{2,\epsilon}
\end{align*}
in $(T +\epsilon e) \bigcap B(O, M)$. Taking $\epsilon \rightarrow 0^+$, we derive $u_1<u_M+ u_2$, and conclude the theorem.
\end{proof}

\bigskip
On the upper half plane, we have explicit solution $x_n^{-\frac{n-2}{2}}$ to \eqref{eq-LN}. Hence if  $S$ lies in the upper half sphere $\mathbb{S}^{n-1}_+$, by the maximum principle, we can prove that $u_S\geq 1$. In general, $u_S>0$ in $S$ by \cite{Loewner&Nirenberg1974}. Hence there is a $\sigma>0$, which only depends on $n, S$, such that,
\begin{align}\label{eq-sigma}
u_S^\frac{4}{n-2}\geq \sigma,
\end{align}
in $S$. 
Then we have the following theorem,
\begin{theorem}\label{thm-r-beta}
Denote $T_M$ as \eqref{eq-TM}, and assume that $T$ is Lipschitz. Let $\sigma$  be a number satisfying \eqref{eq-sigma}.  Then for any $\beta>0$ satisfying
\begin{align}\label{eq-n-beta}
 \beta(n+\beta-2) -\frac{n(n+2)}{4 }\sigma<0,
\end{align}
there is a constant $C>0$, such that
\begin{align*}
|u-u_T| \leq C r^\beta,
\end{align*}
in $T_\frac{M}{2}$.
\end{theorem}
\begin{proof}
First consider the case $k=n$.
For some $B, \beta>0$, $u_T+ B r^\beta$ is a supersolution if 
\begin{align*}
&\Delta (u_T + B r^\beta) -\frac14n(n-2) (u_T + B r^\beta)^{\frac{n+2}{n-2}}<0.
\end{align*}
Using \eqref{eq-polar}, it implies that
\begin{align*}
N(B r^\beta) +L_S(B r^\beta)<r^\frac{n-2}{2} u_S^{\frac{6-n}{n-2}} (B r^\beta)^2 \cdot F (r^\frac{n-2}{2}  u_S^{-1}(B r^\beta)),
\end{align*}
where $F$ is well defined if $|r^\frac{n-2}{2}  u_S^{-1}(B r^\beta)|<1$.
We have
\begin{align*}
&N(B r^\beta) +L_S(B r^\beta) -r^\frac{n-2}{2} u_S^{\frac{6-n}{n-2}} (B r^\beta)^2 \cdot F (r^\frac{n-2}{2}  u_S^{-1}(B r^\beta))\\
&\qquad=B\left( \beta(k+\beta-2) r^\beta- \frac{n(n+2)}{4 }u_S^\frac{4}{n-2} r^\beta   -r^{\frac{n-2}{2}} u_S^{\frac{6-n}{n-2}} B r^{2\beta} \cdot F (r^\frac{n-2}{2}  u_S^{-1}(B r^\beta)) \right)\\
&\qquad<B\left( \beta(k+\beta-2) r^\beta- \frac{n(n+2)}{4 }\sigma r^\beta  \right)\\
&\qquad< 0,
\end{align*}
as $ \beta(k+\beta-2) - \frac{n(n+2)}{4 }\sigma<0$ by the assumption. Here we applied $F (r^\frac{n-2}{2}  u_S^{-1}(B r^\beta))>0.$
So $u_T+ Br^\beta$ is a supersolution. 

To make $u_T- Br^\beta$ a subsolution, first check about the boundary condition.
By Theorem \ref{thm-CSTBD}, $|u-u_T| = |u - r^{-\frac{n-2}{2}}u_S|< C$ in $T_{\frac{M}{2}}$ for any solution $u$. So we can find a $B>0$ such that $u_T -Br^\beta < u$ when $r$ equals a fixed small $r_0<\frac{M}{2}$. In addition, $B$ can be selected such that $Br^\beta< C$ in  $T_{r_0}$ for some different $C$. Here we compare in a smaller $T_{r_0}$ to make sure that $u_T-Br^\beta$ is  positive.

We compute in $T_{r_0}$,
\begin{align*}
&N(-B r^\beta) +L_S(-B r^\beta) -r^\frac{n-2}{2} u_S^{\frac{6-n}{n-2}} (-B r^\beta)^2 \cdot F (r^\frac{n-2}{2}  u_S^{-1}(-B r^\beta))\\
&\qquad=B\left( -\beta(k+\beta-2) r^\beta+ \frac{n(n+2)}{4 }u_S^\frac{4}{n-2} r^\beta   -r^{\frac{n-2}{2}} u_S^{\frac{6-n}{n-2}} B r^{2\beta} \cdot F (r^\frac{n-2}{2}  u_S^{-1}(-B r^\beta))    \right).
\end{align*}
When $r_0$ is sufficiently small,  $\frac{n(n+2)}{4 }u_S^\frac{4}{n-2} r^\beta $ is sufficient to bound the other terms. In fact,
\begin{align*}
\frac{n(n+2)}{4 }u_S^\frac{4}{n-2} r^\beta \geq  \beta(k+\beta-2) r^\beta+ \left(\frac{n(n+2)}{4 } -  \beta(k+\beta-2) \sigma \right) u_S^\frac{4}{n-2} r^\beta,
\end{align*}
and
\begin{align*}
&|r^{\frac{n-2}{2}} u_S^{\frac{6-n}{n-2}} B r^{2\beta} \cdot F (r^\frac{n-2}{2}  u_S^{-1}(-B r^\beta))|\\
&\qquad\leq  C u_S^{\frac{4}{n-2}} r^{\beta}  |r^{\frac{n-2}{2}}  u_S^{-1}B r^\beta \cdot F (r^\frac{n-2}{2}  u_S^{-1}(-B r^\beta))| \\
&\qquad<\left(\frac{n(n+2)}{4 } -  \beta(k+\beta-2) \sigma \right)  u_S^\frac{4}{n-2} r^\beta,
\end{align*}
if $Br^\beta< C,$ and $r< r_0$ is small.  This concludes the case $k=n.$

If $k< n$, fix a $P\in \partial T_\frac{M}{2}$, with $r(P)=0$.  Denote
\begin{align*}
r_P(x)= |x- x(P)|.
\end{align*}
We consider the test function
\begin{align}\label{eq-M-uT}
M_+ = u_T + Br_P^\beta.
\end{align}
in
\begin{align*}
D=\{x\in T_M: 0<r_P^\beta <\delta\},
\end{align*}
for some $\delta$ very small comparing to $n, \beta, M$. Then we set $B$ large, such that $B\delta=(\frac{4}{M})^\frac{n-2}{2}$, which comes from Theorem \ref{thm-CSTBD}.

Notice
\begin{align*}
&r^2 \Delta_{\mathbb{R}^n} (r_P^\beta)- L_S(r_P^\beta)\\&\qquad=r^2\beta(n+\beta-2)r_P^{\beta-2}- \frac{n(n+2)}{4 }u_S^\frac{4}{n-2} r_P^\beta\\
&\qquad \leq \beta(n+\beta-2)r_P^{\beta}- \frac{n(n+2)}{4 }u_S^\frac{4}{n-2} r_P^\beta\\
&\qquad <0.
\end{align*}
 We conclude that \eqref{eq-M-uT} is a supersolution in $D$. 

To complete the maximum principle,
we shift the cone as in Theorem \ref{thm-CSTBD}, and compare the supsolution $u_T+ Br_P^\beta$ with $u$. We derive that $u\leq u_T +B r_P^\beta $  in $D$, implying that $v\leq B r^\beta$ in $ \{x\in D: x^\prime_{n-k}- x^\prime_{n-k}(P)=0 \}$, where $r=r_P$. We can shift $P$ on the set $\{x\in T_\frac{M}{2}: r(x)=0\}$ to derive that $v\leq B r^\beta$ in the domain  $\{x\in T_\frac{M}{2}: B r^\beta(x)<\delta\}$.
If $B r^\beta(x)>\delta$,  $|v|\leq Cr^\beta$ is trivial as we already know that $v$ is uniformly bounded.

For the subsolution $M_- = u_T - Br_P^\beta$, we concern that whether it remains positive in $D$. We check in $D$,
\begin{align*}
 u_T - Br_P^\beta>u_T-B\delta=r^{-\frac{n-2}{2}} u_S-(\frac{4}{M})^\frac{n-2}{2},
\end{align*}
which is positive if $r$ is small. The rest works the same way as $M_+$ .
\end{proof}




\bigskip

\section{Boundary expansions with respect to $d_S$}\label{sec-exp-ds}

\subsection{Expansions of $u_S$}\label{sec-ExpUs}

Recall that $d_S$ is the distance function to $\partial S$ in $S$ , adjusted smooth at points away from $\partial S$. We can trivially extend $d_S$ as a function on $T$.

Notice that at any $x\in T_{M}$ where $d_S(x)$ is small, the distance function  to $\partial T_M$ can be expressed as $d(x) = r(x) \sin (d_S(x) )$. Near any $P \in \partial T$, with $r(P)\neq 0$, we can apply the maximum principle and explicit solutions on the interior balls and on the complement of exterior balls  to show that  $|d^{\frac{n-2}{2}} u_T- 1|< Cd_S$, where $C$ is independent of $r$. See \cite{HanShen1}. It implies 
\begin{align*}
(\sin d_S)^\frac{n-2}{2}u_S(\theta)=d^{\frac{n-2}{2}} u_T = 1+O(d_S),
\end{align*}
as $d_S\rightarrow 0.$

Now $u_S$ satisfies the singular Yamabe type equation \eqref{eq-uS} in smooth $S$. According to \cite{ACF1982CMP}, 
$u_S$ is polyhomogeneous. See also \cite{JiangXiao} or \cite{HanJiang}. 
It means that for any $b\geq n$,
\begin{align}\label{eq-Exp-US}
d_S^\frac{n-2}{2}u_S= v_{S, b}+R_{S,b},
\end{align}
where
\begin{align}\label{eq-Sum-US}
v_{S, b}=1+\sum_{i=1}^{n-1}c_{S, i}d_S^i
+\sum_{i=n}^b\sum_{j=0}^{N_i} c_{S, i, j}d_S^i(\log d_S)^j,
\end{align}
with $c_{S, i}, c_{S, i, j}$'s are independent of $d_S$,
and the remainder $R_{S, b}=O(d_S^{b+\alpha})\cap C^{b+\epsilon}(S)$, for any $0<\epsilon<\alpha<1$. 
This verifies that $d_S^\frac{n-2}{2}u_S$ has a boundary expansion of order $d_S^b$.

\bigskip
\subsection{Expansions of eigenfunctions}
Consider the eigenvalue problem
\begin{align}\label{eq-EigenV}
\Delta_{S} \phi - \frac{n(n+2)}{4} u_S^{\frac{4}{n-2}} \phi= -\lambda \phi,
\end{align}
in $H^1_0$. Notice that
\begin{align*}
u_S^{\frac{4}{n-2}} &= \left(d_S^{-\frac{n-2}{2}}(v_{S,b}+ R_{S,b})\right)^{\frac{4}{n-2}}=d_S^{-2} (v_{S,b}+ R_{S,b})^\frac{4}{n-2}
\end{align*}
where  $\left(v_{S,l}+ R_{S,l}\right)^\frac{4}{n-2}=1+O(d_S)$ and is actually $C^{n-1, \alpha}$ up to $\partial S$ by \eqref{eq-Sum-US}. Denote $\tilde{d}_S =d_S \cdot (v_{S,l}+ R_{S,l})^{-\frac{2}{n-2}}$ as a $C^{n-1,\alpha}$ defining function of $S$, then \eqref{eq-EigenV} is transformed to
\begin{align}\label{eq-PhiI}
\Delta_{S} \phi - \frac{n(n+2)}{4} \frac{ \phi}{\tilde{d}^2_S}= -\lambda \phi.
\end{align}
According to Appendix \ref{sec-Spec},  $\Delta_S - \frac{n(n+2)}{4\tilde{d}_S^2}$ has a complete set of $L^2(S)$-orthonormal eigenfunction's $\{\phi_i\}_{i=1}^\infty$ with respect to the norm $\|\cdot \|_{\tilde{d}_S}$.

By \eqref{eq-PhiI}, an eigenfunction $\phi_i$ satisfies
\begin{align}\label{eq-LSSF}
\Delta_{S} \phi_i - \frac{n(n+2)}{4} \frac{ \phi_i}{d_S^2}= -\lambda_i \phi_i + \frac{n(n+2)}{4} \frac{ \phi_i}{d_S}\cdot \frac{((v_{S,b}+ R_{S,b})^\frac{4}{n-2}-1 )}{d_S},
\end{align}
where $\frac{((v_{S,b}+ R_{S,b})^\frac{4}{n-2}-1 )}{d_S}$ is uniformly bounded. 

Starting with $\phi_i\in H^1_0(S)$, the classical elliptic theory shows that $\phi_i\in C_0^\infty(S)$ and by the maximum principle,
\begin{align*}
\phi_i=O(d_S^\frac{n+2}{2}).
\end{align*}
By \cite{JiangXiao} or \cite{HanJiang}, $d_S^{-\frac{n+2}{2}}\phi_i$ has boundary expansion of order $d_S^b$ for any $b\in \mathbb{N}.$ In fact, for any integer $b\geq n$,
\begin{align}\label{eq-PhiExp}
\phi_i= d_S^\frac{n+2}{2}(\phi_{S, i, b}+R_{S,i, b}),
\end{align}
where
\begin{align}\label{eq-VSLK}
\phi_{S, i, b}(z_S, d_S)=1+\sum_{l=1}^{n-1}c_{S,
i, l}(z_S)d_S^l
+\sum_{l=n}^b \sum_{m=0}^{N_l} c_{S, i, l, m}(z_S)d_S^l(\log d_S)^m,
\end{align}
and 
\begin{align}\label{eq-VSLKRem}
R_{S, i, b}(z_S, d_S)=O(d_S^{b+\alpha})\cap C^{b+\alpha},
\end{align}
 for any $0<\alpha<1$.

\bigskip

 \subsection{Expansions of $v$ with respect to $d_S$}\label{sec-VExpDs} Assume $v= u- u_T$, which is defined on $T_M$ and satisfies \eqref{eq-polar}. By Theorem \ref{thm-CSTBD}, $|v| \leq C_0(M)$ in $T_{\frac{M}{2}}$.
 
 \bigskip
First we derive the estimates of derivatives in the $x^\prime_{n-k}$ direction.
\begin{lemma}\label{lem-v-r-ds}
Assume that  $\beta$ is a number satisfying \eqref{eq-n-beta}. Then for any $q\in \mathbb{N}$,  
 there is a constant $C_q$ depending on $T, M, n, q, S$, such that
\begin{align}\label{eq-xprime-q}
 |D_{x^\prime_{n-k}}^q v|\leq C_q r^\beta d_S^\frac{n+2}{2},
\end{align}
in $T_\frac{M}{2}$.
\end{lemma}
\begin{proof}
We prove by induction. 
The proof of Theorem \ref{thm-r-beta} can be applied to show that $|v|\leq Cr^\beta$ in $T_\frac{3M}{4}$ for some $\beta$ satisfying \eqref{eq-n-beta}.

  Denote $v_q=D_{x^\prime_{n-k}}^q v$.  Inductively we prove that for any $q\in \mathbb{N}$, there is a $\gamma_q\in (\frac{M}{2}, \frac{3M}{4})$, such that  \eqref{eq-xprime-q} holds in $T_{\gamma_q}$, and in addition, for any $m\geq 1$, there is a constant $C_{m, q}$, such that
\begin{align}\label{eq-Dm-vl}
|r^m d_S^m D^m_S v_q| \leq C_{m, q} r^\beta d_S^\frac{n+2}{2}
\end{align} in $T_{\gamma_q}$. Here $D_S$ denotes the derivative with respect to $d_S$ or $z_S$.
When $r\geq\frac{M}{4}$, \eqref{eq-polar} is only sigular when $d_S=0$. We apply the tangential derivative estimates in \cite{JiangXiao} or \cite{HanJiang} to get \eqref{eq-xprime-q}, \eqref{eq-Dm-vl}. In the following, we  show  \eqref{eq-xprime-q}, \eqref{eq-Dm-vl} when $r<\frac{M}{4}$.

Case $q=0$:
Fix a $\delta>0$, which is much smaller than $\frac{M}{4}$. Consider points in $T_{\frac{3M}{4}}$, satisfying
\begin{align}\label{eq-P-delta}
|x^\prime_{n-k}|<\frac{3M}{4}-\delta, \,\, r<\frac{M}{4}.
\end{align}
For any such point $P$, we denote $r_0= r(P)$, and do a scaling $t=r/r_0$, under which the region $\{x\in T_\frac{3M}{4} : \frac{1}{2}r_0 < r(x) < \frac{3}{2}r_0 \}$ is transformed to
\begin{align*}
D=\{(x^\prime_{n-k}, t,\theta):|x_{n-k}^\prime|< \frac{3M}{4},\frac{1}{2}<t< \frac{3}{2},  \theta\in S \},
\end{align*}
 and the equation \eqref{eq-polar} is transformed into \begin{align}\begin{split}\label{eq-polarS}
&\quad  r_0^2 t^2  \Delta_{\mathbb{R}^{n-k}} v+ t^2 v_{tt}+(k-1) t v_t + \Delta_S v-   \frac{n(n+2)}{4} u_S^{\frac{4}{n-2}} v\\
&\qquad =(tr_0)^\frac{n-2}{2}  u_S^{\frac{6-n}{n-2}} v^2 \cdot F \left((tr_0)^\frac{n-2}{2} u_S^{-1}v\right),
\end{split}
\end{align}
where $t^\frac{n-2}{2}$ is smooth in $D$, since $t\in (\frac{1}{2}, \frac{3}{2})$.

By \eqref{eq-Exp-US},
\begin{align}
u_S^{\frac{4}{n-2}}&= d_S^{-2} \left(1+\sum_{i=1}^{n-1}c_{S, i}d_S^i
+\sum_{i=n}^b \sum_{j=0}^{N_i} c_{S, i, j}d_S^i(\log d_S)^j+ R_{S, b}\right)^{\frac{4}{n-2}}\label{eq-USP1}\\
u_S^{\frac{6-n}{n-2}}&= d_S^{-2+\frac{n-2}{2}}\left(1+\sum_{i=1}^{n-1}c_{S, i}d_S^i
+\sum_{i=n}^b \sum_{j=0}^{N_i} c_{S, i, j}d_S^i(\log d_S)^j+ R_{S, b}\right)^{\frac{6-n}{n-2}}\label{eq-USP2}\\
u_S^{-1}&= d_S^{\frac{n-2}{2}} \left(1+\sum_{i=1}^{n-1}c_{S, i}d_S^i
+\sum_{i=n}^b \sum_{j=0}^{N_i} c_{S, i, j}d_S^i(\log d_S)^j+ R_{S, b}\right)^{-1}.\label{eq-USP3}
\end{align}
We write 
\begin{align*}
u_S^{\frac{4}{n-2}} v= \frac{v}{d_S^2} +\frac{v}{d_S}\cdot \frac{(v_{S,b}+ R_{S,b})^\frac{4}{n-2}-1 }{d_S},
\end{align*}
where $\frac{(v_{S,b}+ R_{S,b})^\frac{4}{n-2}-1}{d_S}$ is uniformly bounded.

Now $|v|\leq Cr_0^\beta$ in $D$. We want to prove $|v|\leq Cr_0^\beta d_S^\frac{n+2}{2}$. To this end, first we set $\overline{v}=vd_S$, which vanishes when $d_S=0$, and satisfies a linear equation,
\begin{align}\begin{split}\label{eq-polarS1}
& r_0^2 t^2 \Delta_{\mathbb{R}^{n-k}} \overline{v}+ t^2 \overline{v}_{tt}+(k-1) t \overline{v}_t +\Delta_{S} \overline{v}- 2\frac{(\nabla v, \nabla {d_S})}{d_S}  - \frac{n^2+2n-8}{4 }\cdot \frac{ \overline{v}}{d_S^2}\\
&\qquad = - \frac{n(n+2)}{4} \frac{ v}{d_S}\cdot (d_S^2 u_S^{\frac{4}{n-2}}-1)+v \Delta_S d_S 
\\&\qquad\qquad+ (tr_0)^\frac{n-2}{2} d_S u_S^{\frac{6-n}{n-2}} v^2 \cdot F \left((tr_0)^\frac{n-2}{2} u_S^{-1}v\right),
\end{split}
\end{align}
where right hand side is bounded by $C_1 r_0^\beta d_S^{-\frac{1}{2}}$, where the constant $C_1$ only depends on $u_S, n, C, F$. Here the $ d_S^{-\frac{1}{2}}$ factor comes from the term
\begin{align*}
d_S u_S^{\frac{6-n}{n-2}} \sim d_S^\frac{n-4}{2}= d_S^{-\frac{1}{2}}.
\end{align*}
when $n=3$.

As $P\in T_{\frac{3M}{4}-\delta}$ and $r(P)<\frac{M}{4}$,
We can apply test
functions
\begin{align*}
M_{\pm} = \pm Ar_0^\beta \left(d_S^\frac{4}{3}+(t-t_0)^2 + |x^\prime_{n-k}- x^\prime_{n-k}(P)|^2 \right)
\end{align*}
to \eqref{eq-polarS1}
in
\begin{align*}
D_{\delta, P}:=\{(x^\prime_{n-k}, t, \theta): |t-t_0|<\delta, 0<d_S(\theta)<\delta, |x^\prime_{n-k} - x^\prime_{n-k}(P)|<\delta\},
\end{align*}
 to show that $|\overline{v}(x^\prime_{n-k}, r_0 t, \theta)|\leq M_+$, for some constant $A$ depending on $C, C_1, n, \delta, u_S, F$, but not $r_0$. When $x^\prime_{n-k} = x^\prime_{n-k}(P), t=t_0$, it implies
\begin{align*}
|\overline{v}(x^\prime_{n-k}, r_0 t, \theta)|< A r_0^\beta d_S^\frac{4}{3}.
\end{align*} 
 or equivalently
\begin{align}\label{eq-v-d13}
|v(x^\prime_{n-k}, r_0 t, \theta)|\leq A r_0^\beta d_S^\frac{1}{3}.
\end{align}
Notice that $P$ could be any point satisfying \eqref{eq-P-delta}. Thus $|v|\leq A r^\beta d_S^\frac{1}{3}$ for all points in $ T_{\frac{3M}{4}}$ with $|x^\prime_{n-k}|<\frac{3M}{4}-\delta$.

Next for any point $P$ in $ T_{\frac{3M}{4}-2\delta}$ with $r<\frac{M}{4}$, 
 we  can continue to apply the maximum principle to \eqref{eq-polarS} as in \cite{JiangXiao} or \cite{HanJiang} to show that
\begin{align}\label{eq-v-A-r0}
|v(x^\prime_{n-k}, r_0 t, \theta)| \leq Ar_0^\beta d_S^\frac{n+2}{2}
\end{align} in $D_{\delta, P}$,  for some constant $A$ depending on $C, C_1, n, \delta, u_S, F$.
As $A$ is independent of $r_0$, we have \begin{align}\label{eq-v-A-r}
|v(x^\prime_{n-k}, r, \theta)| \leq Ar^\beta d_S^\frac{n+2}{2},
\end{align}
in $T_\frac{3M}{4} \cap \{x\in T_M: |x^\prime_{n-k}|< \frac{3M}{4}-2\delta\}$. Interior estimates to \eqref{eq-polar} applies to show that, for any $m \in \mathbb{N}$,
\begin{align*}
|r^md_S^m D^m_S v| \leq C_{m, 0}r^\beta d_S^\frac{n+2}{2},
\end{align*}
in $T_\frac{3M}{4} \cap \{x\in T_M: |x^\prime_{n-k}|< \frac{3M}{4}-3\delta\}$, for some constants $C_{m, 0}.$ Then we set $\gamma_0= \frac{3M}{4}-3\delta,$ and derive
\eqref{eq-Dm-vl} for case $q=0$.

If $k=n$, there are no $x_{n-k}^\prime$ coordinates, and we are already done. In the following, we assume that $2\leq k\leq n-1.$

Case $l$ for $l\leq q-1$: assume that \eqref{eq-Dm-vl} is right in $T_{\gamma_l}$ for any $l\leq q-1$. 

For case $q$, applying $ D_{x^\prime_{n-k}}^q$ to \eqref{eq-polar}, we derive an equation of form,
\begin{align}\begin{split}\label{eq-del-vq}
&r^2\Delta_{\mathbb{R}^{n-k}}v_q+ Nv_q +L_{S}  v_q\\
&\qquad-  A_{0, q} r^\frac{n-2}{2} u_S^{\frac{6-n}{n-2}}  v v_q\cdot \tilde{F} (u_T^{-1}v,\cdots, u_T^{-1}v_{q-1})  =H,
\end{split}
\end{align}
where $H$ denotes
\begin{align*}
H := \sum_{0\leq l, m\leq q-1}A_{l, m} r^\frac{n-2}{2} u_S^{\frac{6-n}{n-2}}  v_l v_m\cdot \tilde{F} (u_T^{-1}v, \cdots, u_T^{-1}v_{q-1}), 
\end{align*}
where $A_{l, m}\in \mathbb{N}$. $\tilde{F}$ is an analytic function, well defined if all of its arguments are less than $1$. $H$ is uniformly bounded by induction. Applying \eqref{eq-Dm-vl} for case $q-1$, and  $m=1$, 
$|v_q|\leq C_{1, q-1} r^{\beta-1}d_S^\frac{n}{2}$.

Denote $w_q= r v_q,$ which satisfies  $|w_q|\leq C_{1, q-1} r^{\beta}d_S^\frac{n}{2}$ in $T_{\gamma_{q-1}}$, and vanishes when $r=0$ or $d_S=0$.
By \eqref{eq-del-vq}, $w_q$ satisfies
\begin{align}\begin{split}\label{eq-wq}
&r^2\Delta_{\mathbb{R}^{n-k}}w_q+ N_1 w_q +L_{S}  w_q\\
&\qquad-  A_{0, q} r^\frac{n-2}{2} u_S^{\frac{6-n}{n-2}}  v w_q\cdot \tilde{F} (u_T^{-1}v,\cdots, u_T^{-1}v_{q-1})=r H,
\end{split}
\end{align}
where
\begin{align*}
N_1:= r^2 D_{rr}  +({k-3}) rD_r + 3-k.
\end{align*}
By the induction,  for $0\leq l, m\leq q-1,$
\begin{align*}
r \cdot r^\frac{n-2}{2} u_S^{\frac{6-n}{n-2}}  v_l v_m&=r\cdot r^\frac{n-2}{2}\cdot  O(d_S^{\frac{n-6}{2}}) \cdot O(r^\beta d_S^{\frac{n+2}{2}}) \cdot O(r^{\beta} d_S^{\frac{n+2}{2}} )\\
&= O(r^{\frac{n}{2}+2\beta} d_S^{\frac{3n}{2}-1}),
\end{align*}
in $T_\frac{M}{2}$. Then $r H=O(r^{\frac{n}{2}+2\beta} d_S^{\frac{3n}{2}-1}).$

For any fixed point $P\in \mathbb{R}^n$, we define $r_\rho$ to be
\begin{align*}
\sqrt{\rho |x^\prime_{n-k}-x^\prime_{n-k}(P)|^2 + r^2},
\end{align*}
where $\rho>0$ is a small number such that
\begin{align}\label{eq-rho}
 (\beta+1)(n-k)\rho+ \beta(n+\beta-2) -\frac{n(n+2)}{4 }\sigma,
\end{align}
is negative. 
For any fixed $\beta$ satisfying \eqref{eq-n-beta}, we can find such a $\rho$. For later use, we also denote $r_{\rho^2}=\sqrt{\rho^2 |x^\prime_{n-k}-x^\prime_{n-k}(P)|^2 + r^2}.$ Notice that $r\leq  r_{\rho^2}\leq r_\rho$ as $\rho \in (0,1)$.

We apply the test functions
\begin{align*}
M_\pm = \pm B( r_\rho^{\beta+1} ),
\end{align*}
to \eqref{eq-wq}, in 
\begin{align*}
D_{\delta, P, \rho}:=\{x\in T_\frac{M}{2}: r_\rho^{\beta+1} <\delta\},
\end{align*}
where $P$ is any point in $\mathbb{R}^n$ such that $r(P)=0$
and $D_{\delta, P, \rho} \subseteq T_{\gamma_{q-1}}$.

First on $\partial T \cap \overline{T_{\gamma_{q-1}}}$, $w_q=0$. We set $\delta$ very small comparing to $\rho, \frac{M}{2}$ and $\gamma_{q-1}-\frac{M}{2}$.
In addition,  as $rH=O(r^{\frac{n}{2}+2\beta} d_S^{\frac{3n}{2}-1})$ and
\begin{align*}
A_{0, q} r^\frac{n-2}{2} u_S^{\frac{6-n}{n-2}}  v\cdot \tilde{F} (u_T^{-1}v,\cdots, u_T^{-1}v_{q-1})=O(r^{\frac{n-2}{2}+\beta}d_S^{n-2}),
\end{align*} 
we can set $\delta$ small enough such that
\begin{align}\label{eq-rH}
|rH| <\rho_1 r^{\beta+1}, 
\end{align}
and
\begin{align}\label{eq-A0q}
|A_{0, q} r^\frac{n-2}{2} u_S^{\frac{6-n}{n-2}}  v \cdot \tilde{F} (u_T^{-1}v,\cdots, u_T^{-1}v_{q-1})| <\rho_2,
\end{align}
where $\rho_1, \rho_2$ are small numbers such that
\begin{align}\label{eq-rho12}
 \rho_1+\rho_2+(\beta+1)(n-k)\rho+ \beta(n+\beta-2) -\frac{n(n+2)}{4 }\sigma<0.
\end{align}
Then we set $B$ large such that $B\delta>|w_q|$ at points where $ r_\rho^{\beta+1} =\delta$. Then on $\partial D_\delta$, $|w_q|< M_+$. We plug  $M_+$ into the first three terms of  \eqref{eq-wq},
\begin{align*}
& r^2\Delta_{\mathbb{R}^{n-k}}(r_\rho^{\beta+1} )+ N_1 (r_\rho^{\beta+1} )+L_{S} (r_\rho^{\beta+1} )\\
 &\qquad=r^2(\beta+1)(n-k)\rho r_\rho^{\beta-1}+ r^2(\beta^2-1) r_\rho^{\beta-3} r^2_{\rho^2}\\
 &\qquad\qquad+r^2(\beta+1)(k-2)  r_\rho^{\beta-1} + (3-k) r_\rho^{\beta+1}- \frac{n(n+2)}{4 }u_S^\frac{4}{n-2} r_\rho^{\beta+1}\\
 &\qquad\leq [ (\beta+1)(n-k)\rho +\beta(k+\beta-2)] r_\rho^{\beta+1}- \frac{n(n+2)}{4 }u_S^\frac{4}{n-2} r_\rho^{\beta+1},
\end{align*}
as $r\leq r_\rho, r_{\rho^2}\leq r_\rho$. By \eqref{eq-sigma}, \eqref{eq-rH}, \eqref{eq-A0q}, \eqref{eq-rho12}, $M_+$ is a supersolution  and $w_q\leq M_+$ in $D_{\delta, P, \rho}$. $M_-$ is a subsolution in $D_{\delta, P, \rho}$ for the same reasion.

Then $|v_q|\leq Br_\rho^\beta$ in $D_{\delta, P, \rho}$, in which if $x^\prime_{n-k}=x^\prime_{n-k}(P)$, we derive that $|v_q|\leq Br^\beta$. Recall that in $T_{\gamma_{q-1}}$,
$|v_q|\leq C_{1, q-1} r^{\beta-1}d_S^\frac{n}{2}$, which also implies that $|v_q|\leq Br^\beta$  when $r\geq \delta$, for probably a larger $B$.

In sum, we derive that $|v_q|\leq Br^\beta$ in $T_{\frac{3M}{4}- \rho^{-\frac{1}{2}}\delta}$,

Finally, we use the scaling method and the interior estimates of \eqref{eq-polar}, as in case $q=0$, to show that \eqref{eq-xprime-q}, \eqref{eq-Dm-vl} are right for case $q$ in $T_{\frac{3M}{4}-2 \rho^{-\frac{1}{2}}\delta}$. We set $\gamma_q= \frac{3M}{4}-2 \rho^{-\frac{1}{2}}\delta$ to complete the induction.


\end{proof}

 \bigskip

 We have the following expansion theorem for $v$ with respect to $d_S$. 
\begin{theorem}\label{thm-VExp}
Assume the same assumption as in Theorem \ref{thm-Main}.
Then $d_S^{-\frac{n-2}{2}}( u-u_T)$ has the boundary expansion of order $d_S^b$ for any $b\in \mathbb{N}$, in $T_\frac{M}{2}$.
\end{theorem}

\begin{proof} Set $v= u-u_T$, which satisfies \eqref{eq-polar}. 
In $T_\frac{M}{2}\cap \{r>\frac{M}{4}\}$, as there is only one singular normal direction $d_S$ in the main equation \eqref{eq-polar}, the theorem follows from a stardard arguent for the boundary expansion. A reference is \cite{HanJiang}.

The proof of Lemma \ref{lem-v-r-ds} also implies that \eqref{eq-xprime-q} holds in
in $T_\frac{3M}{4}$.
Around any $P\in T_\frac{3M}{4}$, we denote $r_0= r(P)$, and do a scaling $t=r/r_0$ as in Lemma \ref{lem-v-r-ds}. Then we derive \eqref{eq-polarS}, which has no singularity in the $t$ direction. In the $x^\prime_{n-k}$ direction, we have estimates \eqref{eq-xprime-q}.  By the tangential derivative estimates of \eqref{eq-polarS},  \eqref{eq-del-vq} in $t, z_S$ as in \cite{HanJiang}, we derive that for any $p, q, i, m \in \mathbb{N}$,
\begin{align*}
 | D_t^p D_{x^\prime_{n-k}}^q D_{z_S}^i D_{d_S}^m v|\leq C(T, M, p, q, i, m) r^\beta d_S^{\frac{n+2}{2}-m},
\end{align*}
in $T_{\frac{3M}{4}-\delta}$, where $\delta=\frac{M}{8}$.
Then we can write \eqref{eq-polarS} as an ODE,
\begin{align*}
 v_{d_S d_S} -   \frac{n(n+2)}{4} d_S^{-2} v& =t^\frac{n-2}{2} r_0^\frac{n+2}{n-2} u_S^{\frac{6-n}{n-2}} v^2 \cdot F (t^\frac{n+2}{n-2} r_0^\frac{n-2}{2} u_S^{-1}v)\\
&\qquad -(r_0^2 t^2  \Delta_{\mathbb{R}^{n-k}} v+ t^2 v_{tt}+(n-1) t v_t )\\
&\qquad +  \frac{n(n+2)}{4} \frac{v}{d_S}\cdot \frac{d_S^2 u_S^{\frac{4}{n-2}} -1 }{d_S} +(
 v_{d_S d_S} - \Delta_S v).
\end{align*}

Then the ODE iteration in \cite{HanJiang} applies to prove that $v$ has a boundary expansion, that for any integer $b\geq n$, 
\begin{align}\label{eq-VExp1}
v= d_S^\frac{n+2}{2}(v_{S, b}+R_{S, v, b}),
\end{align}
in $T_\frac{M}{2}$,
where
\begin{align}\label{eq-VVSLK1}
v_{S, b}&=\sum_{i=0}^{n-1}c_{S,
v, i}(x^\prime_{n-k}, r, z_S)d_S^i
+\sum_{i=n}^b \sum_{j=0}^{N_i} c_{S, v, i, j}(x^\prime_{n-k}, r, z_S)d_S^i(\log d_S)^j,
\end{align}
and it holds, for any integers $l, m, p, q \in \mathbb{N}$, and $\alpha \in (0,1)$, 
\begin{align}\begin{split}\label{eq-VVSRem1}
|r^p D_r^p D_{x_{n-b}^\prime}^q   D_{z_S}^l c_{S, v, i}|&\leq C(T, M, l, i, p, q).\\
|r^p D_r^pD_{x_{n-b}^\prime}^q D_{z_S}^l c_{S, v, i, j}|&\leq C(T, M,  l, i, j, p, q).\\
|r^p D_r^pD_{x_{n-b}^\prime}^q D_{z_S}^l D_{d_S}^m R_{T, M , b}|&\leq C(T, M,  l, m, p, q, \alpha) d_S^{b+\alpha-m}.
\end{split}
\end{align}
As in Theorem \ref{eq-Dm-vl}, in the estimates \eqref{eq-VVSRem1}, we can have an extra factor $r^\beta$ on the right hand side. But we do not need it in the following sections.
\end{proof}

\bigskip

\section{Eigenvalue Growth Estimate}\label{sec-EigenG}
In this section, we prove Theorem \ref{thm-Eigen-Growth}.
\begin{proof}
By a standard argument using the Lax-Milgram Theorem as in Appendix \ref{sec-Spec}, we have the first part of Theorem \ref{thm-Eigen-Growth}.

Now denote $A_\lambda(x,y )=\sum_{\lambda_i \leq \lambda} \phi_i(x)\phi_i(y)$, where
$
\phi_i
$'s are eigenvectors of $L-\frac{\kappa}{d^2}$ corresponding to $\lambda_i$.
Simple calculation shows that
\begin{align*}
(\Delta-\frac{\kappa}{d^2})A_\lambda (x, x)&=-2 \sum_{\lambda_i \leq \lambda} \lambda_i \phi_i^2 +\sum_{\lambda_i \leq \lambda}\frac{\kappa}{d^2}\phi_i^2 +2 \sum_{\lambda_i \leq \lambda} g^{ml}\partial_l \phi_i \partial_m \phi_i\\
&\geq -2\lambda A_\lambda(x, x).
\end{align*}
Assume $A_\lambda(x,x) \leq M$ on $\{d=d_0\}$ for $d_0=\min\{d_1,  \sqrt{\frac{\kappa}{4\lambda}}\}$, where $d_1$ is the inner radius of $\partial S$. $A_\lambda(x,x) = 0$ on $\partial S$.
Then the maximum principle implies that
\begin{align}\label{eq-Axx1}
|A_\lambda(x, x)| \leq M,
\end{align}
in $\{0<d<d_0\}$.

\bigskip
 If  $u,v$ satisfy
\begin{align}\label{eq-uv}
\Delta u-\frac{\kappa}{d^2}u=v,
\end{align}
in a ball $B_g(x_0, r) \subseteq S$, we multiply \eqref{eq-uv} by $d(x_0)^2$ and do the scaling
\begin{align*}
\bar{x}= (x - x_0)/d(x_0),
\end{align*}
to transform the equation to
\begin{align*}
a_{ij}( \bar{x}) u_{\bar{i}\bar{j}}+b_i(\bar{x}) u_{\bar{i}}-\frac{\kappa d(x_0)^2}{d(d(x_0)\bar{x})^2} u =   d(x_0)^2 v,
\end{align*}
for some smooth functions $a_{ij}, b_i$.
Notice $\frac{\kappa d(x_0)^2}{d(d(x_0)\bar{x})^2}$ is bounded and smooth in $\bar{x}$ by \eqref{eq-Eigen-Cond}.
Then by the $W^{2,2}$ estimate, for $j\geq 2$,
\begin{align}\label{eq-W22g}
||u||_{W^{2,2}_G(B_G(x_0, 2^{-j}))}\leq C_j\left(||P u||_{L^2_G(B_G(x_0, 2^{-j+1}))}+ ||u||_{L^2_G(B_G(x_0, 2^{-j+1}))}\right),
\end{align}
where $G$ denotes the scaled metric $d(x_0)^{-2} g$, and
\begin{align*}
P:=a_{ij}( \bar{x})\partial_{\bar{i}\bar{j}}+ b_i(\bar{x})\partial_{\bar{i}}-\frac{\kappa d(x_0)^2}{d(d(x_0) \bar{x})^2}.
\end{align*}

We prove by induction that for any $1\leq m\leq p$,
\begin{align}\label{eq-W2l2}
||u||_{W^{2m,2}_G(B_G(x_0, 2^{-2m}))}\leq C_m  \sum_{q=0}^{m} \|P^{q}u\|_{L^{2}_G(B_G(x_0, 2^{-1}))},
\end{align}
where $C_m$ is independent of $u$.
$p=1$ is by \eqref{eq-W22g}. Assume that case $p-1$ is right. We prove case $p$.
\begin{align*}
&\qquad\|u\|_{W^{2p-1,2}_G(B_G(x_0, 2^{-2p+1}))}\\
&\leq C\left(\|P D_{\bar{x}}^{2p-3} u\|_{L^{2}_G(B_G(x_0, 2^{-2p+2}))} +  \| D_{\bar{x}}^{2p-3} u\|_{L^{ 2}_G(B_G(x_0, 2^{-2p+2}))}\right)\\
&\leq C( \| D_{\bar{x}}^{2p-3} P u\|_{L^{2}_G(B_G(x_0,2^{-2p+2}))} +
\tilde{C}\| D_{\bar{x}}^{2p-2}  u\|_{L^{2}_G(B_G(x_0, 2^{-2p+2}))} \\
&\qquad+
 \tilde{C} \| D_{\bar{x}}^{2p-3} u\|_{L^{ 2}_G(B_G(x_0, 2^{-2p+2}))} )\\
 &\leq C_{2p}  \sum_{q=0}^{p} \|P^{q}u\|_{L^{2}_G(B_G(x_0, 2^{-1}))},
\end{align*}
by induction.
Similar estimates hold for $\|u\|_{W^{2p,2}_G(B_G(x_0, 2^{-2p}))}$ and we conclude \eqref{eq-W2l2}.

Fix $p$ as the smallest integer such that $2p>\frac{l}{2}$, where we recall that $l=\dim S$.
By Lemma 17.5.2 in H\"omander \cite{Hormander}, applying a cutoff function, we can show that
\begin{align}\label{eq-xi}
\xi^{p-\frac{l}{4}}|u(x_0)|\leq C\left(\|u\|_{W^{2p, 2}_G(B_G(x_0, 2^{-2p}))}+\xi^{p} \|u\|_{L^{2}_G(B_G(x_0, 2^{-1}))} \right),
\end{align}
for any $\xi \geq 1.$ Here $C$ is independent of $x_0$. Fix $\xi= \max\{\lambda d(x_0), 1\}$
and set $u(x)= A_\lambda(x, y)$. We derive that, by \eqref{eq-xi}, \eqref{eq-W2l2},
\begin{align*}
|A_\lambda(x_0, y)|&\leq C\left(\xi^{-p+\frac{l}{4}} \|u\|_{W^{2p, 2}_G(B_G(x_0, 2^{-2p}))}+\xi^{\frac{l}{4}} \|u\|_{L^{2}_G(B_G(x_0, 2^{-1}))} \right)\\
&\leq C_{2p} \left(  \xi^{-p+\frac{l}{4}} \sum_{q=0}^{p} \|P^{q}A_\lambda(x, y)\|_{L^{2}_G(B_G(x_0, 2^{-1}))} + \xi^{\frac{l}{4}} \|u\|_{L^{2}_G(B_G(x_0, 2^{-1}))} \right),\\
&\leq  \tilde{C}_{2p} \xi^\frac{l}{4} ||A_\lambda(x, y)\|_{L^{2}_G(B_G(x_0, \frac{1}{2}))}\\
&\leq  \tilde{C}_{2p} \xi^\frac{l}{4} d(x_0)^{-\frac{l}{2}}  ||A_\lambda(x, y)\|_{L^{2}_g(S, x)},
\end{align*}
 where $L^2_g(S, x)$ denotes that the $L^2$ norm is calculated with respect to $g$ in variable $x$. Here we applied that $\|P^{q}A_\lambda(x, y)\|_{L^{2}_G(B_G(x_0, 2^{-1}))}\leq C\xi^q \|A_\lambda(x, y)\|_{L^{2}_G(B_G(x_0, 2^{-1}))}$ where $C$ is independent of $\lambda$. By the fact,
\begin{align*}
A_\lambda(x, y) =(A_\lambda(x, z), A_\lambda(y, z))_{L^2_g(S, z)},
\end{align*}
we have
\begin{align*}
\|A_\lambda(x, z)\|^2_{L^2_g(S, z)}=A_\lambda(x, x) \leq   \tilde{C}_{2p} \xi^\frac{l}{4} d(x)^{-\frac{l}{2}}  ||A_\lambda(x, y)\|_{L^{2}_g(S, y)},
\end{align*}
i.e.
\begin{align*}
\|A_\lambda(x, z)\|_{L^2_g(S, z)}\leq  \tilde{C}_{2p} \xi^\frac{l}{4} d(x)^{-\frac{l}{2}},
\end{align*}
which further implies
\begin{align*}
A_\lambda(x, x)\leq  C(p)  \xi^\frac{l}{2} d(x)^{-l}.
\end{align*}

Now for points in $\{d(x)\geq d_0\}$, 
we have that if $\lambda d(x)^2\leq 1$, then
\begin{align}\label{eq-Axx2}
A_\lambda(x, x)\leq    C(l)   d_0^{-l},
\end{align}
 and if $\lambda d(x)^2>1$,
\begin{align*}
A_\lambda(x, x)\leq C(l)\lambda^{\frac{l}{2}}.
\end{align*}
Recall that $d_0 \sim \sqrt{\frac{\kappa}{2}}\lambda^{-\frac{1}{2}}$ as $\lambda$ large, so
 in both cases, we derive
\begin{align*}
A_\lambda(x, x)\leq   C(\kappa, d_1, l) \lambda^{\frac{l}{2}},
\end{align*}
which also implies
\begin{align}\label{eq-PhiUpBd}
 |\phi_i|\leq C(\kappa, d_1, l) \lambda_i^\frac{l}{4}.
\end{align}

For points in $\{0<d(x)\leq d_0\}$, applying \eqref{eq-Axx1} with $M= C(\kappa, d_1, l)\lambda^{\frac{l}{2}}$, we have
\begin{align*}
|A_\lambda(x,x)|\leq {C}(\kappa, d_1, l) \lambda^\frac{l}{2}.
\end{align*}

Finally, the number of eigenvalues with multiplicity counted is
\begin{align*}
N(\lambda) =\int_S A_\lambda(x, x) dx\leq C \lambda^\frac{l}{2},
\end{align*}
which implies
\begin{align}\label{eq-LamdaiEstm}
\lambda_i >C i^\frac{2}{l}. 
\end{align}

Then we finish the proof of Theorem \ref{thm-Eigen-Growth}.
\end{proof}
We remark that by the assumption of Theorem \ref{thm-Main} and \eqref{eq-USP1},  $d$ actually is
\begin{align*}
d(\theta)= u_S^{-\frac{2}{n-2}}= d_S (1+\sum_{i=1}^{n-1}c_{S, i} d_S^i
+\sum_{i=n}^b \sum_{j=0}^{N_i} c_{S, i, j} d_S^i(\log d_S)^j+ R_{S, b})^{-\frac{2}{n-2}},
\end{align*}
which satisfies \eqref{eq-Eigen-Cond}.

\bigskip

 \section{Proof of Theorem \ref{thm-Main}}\label{sec-SmoothS}
In this section, we prove Theorem \ref{thm-Main}. Assume the same assumption as in Theorem \ref{thm-Main}.

According to Lemma \ref{lem-A1}, $L_S$ has a complete set of $L^2(S)$-orthonormal eigenfunctions $\{\phi_i\}_{i=1}^\infty$, with corresponding eigenvalues $\{-\lambda_i\}$, where $0<\lambda_i \leq \lambda_j$ if $i<j$.
Denote
\begin{align}
\overline{m}_i&=\frac{-(k-2)+\sqrt{(k-2)^2+4\lambda_i}}{2}>0, \nonumber\\
\underline{m}_i&=\frac{-(k-2)-\sqrt{(k-2)^2+4\lambda_i}}{2}<0 \nonumber
,\end{align}
as the zeros of
\begin{align*}
m^2 +(k-2)m -\lambda_i=0.
\end{align*}
Then
\begin{align}\label{eq-mupiMmlwi}
\overline{m}_i \sim 2\sqrt{\lambda_i},\quad \underline{m}_i\sim - 2\sqrt{\lambda_i}
\end{align}
as $\lambda_i \rightarrow \infty.$

Denote $I$ as monoid of $\mathbb{R}$, which is generated by $\{2, \frac{n-2}{2}, \overline{m}_1, \overline{m}_2, \cdots\}$. Notice $0\in I$.

We define the index set $J$ as the following: first, $\{\overline{m}_1, \overline{m}_2, \cdots\} \subseteq J$;
second, if $a, b, c\in J$, then
 $ \frac{n-2}{2}+ a+b + l(\frac{n-2}{2}+ c)\in J$ for any $l \in \mathbb{N}$, and if $k\neq n$, we also request $a+2\in J.$ This is derived from the formal computation of \eqref{eq-polar}.

Easy to see that $J$ is a subset of $I$. Align the elements in $I, J$ in the ascending order. For any element $a\in I$, we denote $a^+$ the next element in $I$, and denote $a^-$ the largest number in $I$ that is smaller than $a$. 

In this section, we prove  Theorem \ref{thm-Main} with the index set $I$. Then the expansion exists, and we can apply the formal computation to show that the index set can be reduced to $J$.

We remark that $\overline{m}_1$ is the smallest element in $J$. If  $S$ lies in the upper half sphere $\mathbb{S}^{n-1}_+$, then $u_S\geq 1$ by the maximum principle, which implies that $\lambda_1\geq \frac{n(n+2)}{4}$, i.e., 
\begin{align*}
\overline{m}_1\geq \frac{n(n+2)}{2\left(\sqrt{(n-2)^2 +n(n+2)} +n-2\right)}.
\end{align*}
When $n=3$, it implies that $\overline{m}_1\geq \frac{3}{2}$ and $v=O(r^\frac{3}{2}d_S^\frac{n+2}{2})$.

We have the following lemma,
\begin{lemma}\label{lem-Fw1wk}
Assume that for some $b\in \mathbb{N}$, $w_1, \cdots, w_l$ are functions defined in $S$ that have boundary expansions of order $d_S^b$, and when $d_S=0$, $(w_1,\cdots, w_l)= (a_1, \cdots, a_l)$.
In addition, we assume that $G$ is a function with $l$ variables which is smooth around  $ (a_1, \cdots, a_l) \in \mathbb{R}^l$. Then $G(w_1, \cdots, w_l)$ has a boundary expansion of order $d_S^b$.
\end{lemma}
The proof is by formal computation.

An example is that by \eqref{eq-Exp-US}, for any $b \in \mathbb{N}$,
\begin{align*}
d_S^{2-\frac{n-2}{2}} u_S^{p-2}=   (d_S^\frac{n-2}{2} u_S)^{p-2},
\end{align*} 
 has expansion of order $d_S^b$, since when $d_S=0$, $d_S^\frac{n-2}{2} u_S=1$, and $G(w)= w^{p-2}$ is a smooth function when $w$ is around $1$.

\vskip0.2in

We define a sequence of integers $\{\tilde{N}_i\}$ in following way: For  $i\in I$, we formally compute $\tilde{N}_i$ successively.

(1) Start with $i=0$, and set $\tilde{N}_0=0$; 
 
(2) Take the next larger $i\in I$. If $i\neq \overline{m}_l$ for any $l\geq 1$, $\tilde{N}_i$ equals $j-1$ where $j$ is the smallest integer such that the term $r^i (\log r)^j $ in
\begin{align}\label{eq-Fk-1}
r^\frac{n-2}{2} u_S^{\frac{6-n}{n-2}} v_{i^-}^2 \cdot F (r^{\frac{n-2}{2}} u_S^{-1}v_{i^-}) -r^2\Delta_{\mathbb{R}^{n-k}}v_{i^-}
\end{align}
has zero coefficient. Here $v_{i^-}:= \sum_{l\in I, l\leq i^-}\sum_{j=0}^{\tilde{N}_l} \tilde{c}_{l,j}(x^\prime_{n-k}, \theta) r^l(\log r)^j $ with undetermined smooth functions $\tilde{c}_{l,j}(x^\prime_{n-k}, \theta)$;

(3)  If $i= \overline{m}_l$ for some $l\geq 1$, $\tilde{N}_i$ equals $j$, which is the smallest integer such that the term $r^i (\log r)^j $ in
\eqref{eq-Fk-1}
has zero coefficient. Go to step (2) to compute $\tilde{N}_i$ for a larger $i\in I$.

Here we do not need to know the exact value of $\tilde{c}_{l,j}$'s, but only set them to be unknowns, and do formal computation to get $\tilde{N}_i$.

\bigskip

The following technical lemma plays a key role in the proof of Theorem \ref{thm-Main}.
\begin{lemma} \label{lem-IntF}
Fix an index $A>0$.
Assume that on $T_\frac{M}{2}$, we have a function 
\begin{align}\label{eq-F-w}
F(x_{n-k}^\prime, r, \theta)= r^a (\log r)^j\cdot w(x_{n-k}^\prime, r, \theta),
\end{align}
for some numbers $a \in \mathbb{R}^+, j\in \mathbb{N}$, such that
$
a\leq A,\, j\leq \tilde{N}_a.
$
If $a<A$, we assume that $w$ only depends on $x_{n-k}^\prime, \theta.$
In addition, we assume that $d_S^{-\frac{n+2}{2}} w$ has an expansion of order  $d_S^b$ for any $b\in \mathbb{N}$, and it holds in $T_\frac{M}{2}$, that
for any $p, q, l, m\in \mathbb{N}$,
\begin{align}\label{eq-w-pqlm}
r^p d_S^m \left\lvert  D_r^p D_{x^\prime_{n-k}}^q  D_{z_S}^l D_{d_S}^m (d_S^{-\frac{n+2}{2}} w) \right\rvert&\leq C(T, M, p, q, l, m).
\end{align}
 Denote
\begin{align*}
F_i= \int_S F \cdot \phi_i d\theta.
\end{align*}
 Then the following  terms, with $r_0=\frac{M}{2}$,
\begin{align*}
&H_1= \sum_{i=1}^\infty \frac{r_0^{\underline{m}_i-\overline{m}_i}  r^{\overline{m}_i} \phi_i}{\overline{m}_i-\underline{m}_i}\int_0^{r_0} r^{-1-\underline{m}_i}F_i dr \\
&H_2= \sum_{i=1}^\infty \frac{r^{\underline{m}_i} \phi_i}{\overline{m}_i-\underline{m}_i}\int_0^{r} r^{-1-\underline{m}_i}F_i dr \\
&H_3 =\sum_{i=1}^\infty  \frac{r^{\overline{m}_i} \phi_i}{\overline{m}_i-\underline{m}_i}\int_r^{r_0} r^{-1-\overline{m}_i}F_i dr,
\end{align*}
have expansions of form,
\begin{itemize}

\item
  if $a=A$,
\begin{align}\label{eq-HExp1}
\sum_{l\in I, l<A}  H_{l, 0}(x_{n-k}^\prime,  \theta) r^l  + \sum_{m=0}^{j+1} H_{A, m}(x_{n-k}^\prime,  r, \theta) r^A (\log r)^m,
\end{align}
\item if $a<A$,
\begin{align}\begin{split}\label{eq-HExp2}
&\sum_{l\in I, l<a}  H_{l, 0}(x_{n-k}^\prime, \theta) r^l  + \sum_{m=0}^{j+1} H_{a, m}(x_{n-k}^\prime, \theta) r^a (\log r)^m \\
&\qquad+\sum_{l\in I, a<l<A}  H_{l, 0}( x_{n-k}^\prime, \theta) r^l  + H_{A, 0}(x_{n-k}^\prime,  r, \theta) r^A,
\end{split}
\end{align}
\end{itemize}
where all coefficients $H_{l, m}$'s  satisfy, that for any fixed $x_{n-k}^\prime,  r$,
\begin{align}\label{eq-HlmL2}
||H_{l, m}||_{L^2(S)}&\leq C(T, M,l,m),\quad 
||L_S (H_{l, m}) ||_{L^2(S)}\leq C(T, M,l,m),
\end{align}
where $C$ is independent of $x_{n-k}^\prime,  r$.
Here $H_{a, j+1}$ is not a zero function only when $a=\overline{m}_i$ for some $i$.

In addition, for any $p, q\in \mathbb{N}$, we have in $T_\frac{M}{2}\cap \{r<\frac{M}{4}\}$,
\begin{align}\begin{split}\label{eq-Hlm-mp}
||r^p D^p_r D_{x_{n-k}^\prime}^q  (H_{l, m})||_{L^2(S)}&\leq C( T, M,l, m,p, q),\\
||r^p D^p_r D_{x_{n-k}^\prime}^q L_S (H_{l, m}) ||_{L^2(S)}&\leq C(T, M,l,m,p, q),\end{split}
\end{align}

\end{lemma}
\begin{proof}
First we show the expansion and \eqref{eq-HlmL2}.
Notice that $F_i$'s only depend on $x_{n-k}^\prime, r$. Denote the operator $T=(-L_S)^\frac{1}{2}$ as in Appendix \ref{sec-lem}.
For any integer $N>0$, if $T^N F \in L^2(S)$, we have
\begin{align*}
F_i &= \frac{1}{\lambda_i^\frac{N}{2}} \int_S F\cdot ( T^N \phi_i) d\theta \\
&=\frac{1}{\lambda_i^\frac{N}{2}} \int_S (T^N F)\cdot \phi_i d\theta.
\end{align*}
By the assumption, $d_S^{-\frac{n+2}{2}} w$ has an expansion of order  $d_S^b$ for any $b\in \mathbb{N}$.
First if $n$ is even, we set $N=\frac{n+2}{2}$. Then if $N$ is even, $T^N F= (-L_S)^\frac{N}{2} F \in L^2(S)$, and if $N$ is odd,  $T^{N-1} F= (-L_S)^\frac{N-1}{2} F =O(d_S)$. By lemma \ref{lem-Tw}, $T^N F\in L^2(S)$.
Secondly if $n$ is odd, we set $N= \frac{n+1}{2}$. Then $T^N F \in L^2(S)$ for the same reason. In sum, we have
\begin{align}\label{eq-FiBd}
|F_i| \leq C(w, S) r^a (\log r)^j \cdot \lambda_i^{-\frac{N}{2}}.
\end{align}

Then we look into the three integrals $H_1, H_2$ and $H_3$. 
In fact, $H_1$ is already of form \eqref{eq-HExp1}, \eqref{eq-HExp2}, as
\begin{align}\label{eq-H1-exp}
&H_1= \sum_{\overline{m}_i<A} \left(\frac{r_0^{\underline{m}_i-\overline{m}_i}  \phi_i}{\overline{m}_i-\underline{m}_i}\int_0^{r_0} r^{-1-\underline{m}_i}F_i dr \right) \cdot  r^{\overline{m}_i}+ H_{A, 0}(x_{n-k}^\prime, r, \theta)r^A,
\end{align}
where
\begin{align}\label{eq-H1-HA0}
H_{A, 0}(x_{n-k}^\prime, r, \theta): = \sum_{\overline{m}_i\geq A} \frac{r_0^{\underline{m}_i-\overline{m}_i}  r^{\overline{m}_i-A} \phi_i}{\overline{m}_i-\underline{m}_i}\int_0^{r_0} r^{-1-\underline{m}_i}F_i dr.
\end{align}
It is clear that there are only finite many terms with $\overline{m}_i<A$ in \eqref{eq-H1-exp}, and their estimates are straightforward. We only have to worry about $H_{A, 0}$.
While applying \eqref{eq-FiBd} to estimate \eqref{eq-H1-HA0},  the integration above will produce a factor at scale $\frac{1}{\underline{m}_i}$, when $i$ gets large. By \eqref{eq-mupiMmlwi},  $\frac{1}{\overline{m}_i-\underline{m}_i} \cdot \frac{1}{\underline{m}_i}$ contributes to an additional $\lambda_i^{-1}$ factor
. Thus by \eqref{eq-FiBd}, \eqref{eq-H1-HA0},
\begin{align*}
&\left\lVert L_S\left(\sum_{\overline{m}_i\geq A} \frac{r_0^{\underline{m}_i-\overline{m}_i}  r^{\overline{m}_i-A} \phi_i}{\overline{m}_i-\underline{m}_i}\int_0^{r_0} r^{-1-\underline{m}_i}F_i dr \right)\right\rVert^2_{L^2(S)}\\
&\qquad = \left\lVert\sum_{\overline{m}_i\geq A} \frac{r_0^{\underline{m}_i-\overline{m}_i}  r^{\overline{m}_i-A} \lambda_i\phi_i}{\overline{m}_i-\underline{m}_i}\int_0^{r_0} r^{-1-\underline{m}_i}F_i dr\right\rVert^2_{L^2(S)}\\
&\qquad \leq  \sum_{\overline{m}_i\geq A} \left\lvert  \frac{r_0^{\underline{m}_i-A}  \lambda_i}{\overline{m}_i-\underline{m}_i}\int_0^{r_0} r^{-1-\underline{m}_i}F_i dr \right\rvert^2 \\
&\qquad \leq  \sum_{ \overline{m}_i\geq A} \left( C(w, S, A, M, j) \lambda_i^{-\frac{N}{2}}\right)^2,
\end{align*}
where $\sum_{\overline{m}_i\geq A} \lambda_i^{-N}$ is convergent as
\begin{align*}
-N \cdot \frac{2}{k-1} \leq  -\frac{n+1}{2} \cdot \frac{2}{k-1} <-1,
\end{align*}
and by Theorem \ref{thm-Eigen-Growth},
\begin{align}\label{eq-converge}
\sum_{\overline{m}_i\geq A} \lambda_i^{-N} \leq \sum_{\overline{m}_i\geq A} i^{-N\cdot \frac{2}{k-1}},
\end{align}
is convergent. 
So we derive \eqref{eq-HlmL2} for $H_1$.

The discussion of $H_2$ is similar. We have
\begin{align*}
&L_S\left( \sum_{i=1}^\infty \frac{r^{\underline{m}_i} \phi_i}{\overline{m}_i-\underline{m}_i}\int_0^{r} r^{-1-\underline{m}_i}F_i dr  \right)= \sum_{i=1}^\infty \frac{r^{\underline{m}_i} \lambda_i\phi_i}{\overline{m}_i-\underline{m}_i}\int_0^{r} r^{-1-\underline{m}_i}F_i dr.
\end{align*}
The difference here is that all terms are of order $r^a (-\log r)^m$ for some $0\leq m\leq j$. 
\begin{itemize}

\item
 if $a= A$, we write $H_2$ as $H_{A, j}r^A (\log r)^j$. We estimate $H_{A, j}$, by \eqref{eq-FiBd},
\begin{align*}
&\left\lVert L_S\left( \sum_{i=1}^\infty \frac{r^{\underline{m}_i-A}(\log r)^{-j} \phi_i}{\overline{m}_i-\underline{m}_i}\int_0^{r} r^{-1-\underline{m}_i}F_i dr \right) \right\rVert^2_{L^2(S)}\\
&\qquad = \left\lVert \sum_{i=1}^\infty \frac{r^{\underline{m}_i-A}(\log r)^{-j} \lambda_i\phi_i}{\overline{m}_i-\underline{m}_i}\int_0^{r} r^{-1-\underline{m}_i}F_i dr \right\rVert^2_{L^2(S)}\\
&\qquad \leq   \sum_{i=1}^\infty \left\lvert  C(w, S) \lambda_i^{-\frac{N}{2}}\cdot\frac{r^{\underline{m}_i-A}(\log r)^{-j} \lambda_i}{\overline{m}_i-\underline{m}_i}\int_0^{r} r^{A-1-\underline{m}_i}(\log r)^j dr \right\rvert^2 \\
&\qquad \leq  \sum_{i=1}^\infty  C(w, S, A, M, j)  \lambda_i^{-N},
\end{align*}
which converges as \eqref{eq-converge}.

\item if $a<A$, by the assumption, $w$ only depends on $x_{n-k}^\prime, \theta.$ Then the integration has explicit formula, and produces terms of order  $ r^{a -\underline{m}_i}  (- \log r)^l$, for $0\leq l\leq j$. For their coefficients, we can estimate in a similar way as in case $a=A$.
\end{itemize}

For $H_3$, notice when $\overline{m}_i=a$, $\int r^{-1-\overline{m}_i} \cdot r^a (\log r)^j dr= \frac{1}{j+1}(\log r)^{j+1}$. So we may have a term of order $r^a (\log r)^{j+1}$ in the expansion of $H_3$. 
In $H_3$, for terms with $\overline{m}_i\geq A$, we just apply \eqref{eq-FiBd} to show that

\begin{itemize}
\item if $a=A$,  \begin{align}\label{eq-H3-mi>A}
\sum_{\overline{m}_i\geq A} \frac{r^{\overline{m}_i} \phi_i}{\overline{m}_i-\underline{m}_i}\int_r^{r_0} r^{-1-\overline{m}_i}F_i dr
\end{align}
can be written as,
\begin{align*}
 H_{A, m}(x_{n-k}^\prime, r, \theta) r^A (\log r)^{m},
\end{align*}
where $m=j+1$ if $A=\overline{m}_i$ for some $i$; otherwise $m=j$. 
\item if $a<A$, as $w$ only depends on $x_{n-k}^\prime, \theta$, \eqref{eq-H3-mi>A} can be written as
\begin{align*}
\sum_{m=0}^{j} H_{a, m}(x_{n-k}^\prime, \theta) r^a (\log r)^m + H_{A , 0}(x_{n-k}^\prime, r, \theta) r^A ,
\end{align*}
\end{itemize}
where all coefficients $H_{l,m}$'s satisfy \eqref{eq-HlmL2}. 

For terms with $\overline{m}_i<A$,
\begin{itemize}
\item
if $a=A$,
\begin{align*}
\sum_{\overline{m}_i< A} \frac{r^{\overline{m}_i} \phi_i}{\overline{m}_i-\underline{m}_i}\int_r^{r_0} r^{-1-\overline{m}_i}F_i dr &=\sum_{\overline{m}_i< A} \frac{r^{\overline{m}_i} \phi_i}{\overline{m}_i-\underline{m}_i}\int_0^{r_0} r^{-1-\overline{m}_i}F_i dr\\
&\qquad \qquad - \sum_{\overline{m}_i< A} \frac{r^{\overline{m}_i} \phi_i}{\overline{m}_i-\underline{m}_i}\int_0^{r} r^{-1-\overline{m}_i}F_i dr,
\end{align*}
which can be dealt with in the same way as for $H_1$ and $H_2$. These are only finite many terms, and the estimates of the coeffcients are straightforward.

\item if $a<A$, $w$ only depends on $ z^\prime,\theta$, and we still derive \eqref{eq-HlmL2} by  \eqref{eq-FiBd} and the explicit integral formula of $r^{-1-\overline{m}_i} \cdot r^a (\log r)^j.$
\end{itemize}

\bigskip
Next, we prove \eqref{eq-Hlm-mp}. Applying $D_{x_{n-k}^\prime}^q$ to $H_1, H_2, H_3$, as only $w$ depends on $x_{n-k}^\prime$, we can use the same arguments above and \eqref{eq-w-pqlm}, to derive  \eqref{eq-Hlm-mp} for case $p=0.$

For $rD_r$, as $H_{l, m}$ is independent of $r$ if $l\neq A$, so we only need to consider $rD_r H_{A, m}$. The only trouble is that $rD_r (r^{\overline{m}_i})= \overline{m}_i r^{\overline{m}_i}$, which  produces an extra  factor $\overline{m}_i$.

For \eqref{eq-H1-HA0} in $H_1$, as in $T_\frac{M}{2}\cap \{r<\frac{M}{4}\}$, $r<\frac{1}{2}r_0$, so we compute
\begin{align}\begin{split}\label{eq-rDr-p}
&\left\lVert  (rD_r)^p L_S \left(\sum_{\overline{m}_i\geq A} \frac{r_0^{\underline{m}_i-\overline{m}_i}  r^{\overline{m}_i-A} \phi_i}{\overline{m}_i-\underline{m}_i}\int_0^{r_0} r^{-1-\underline{m}_i}F_i dr \right)\right\rVert^2_{L^2(S)}\\
&\qquad = \left\lVert\sum_{\overline{m}_i\geq A} \frac{ (\overline{m}_i-A)^p\lambda_i   r_0^{\underline{m}_i-\overline{m}_i}  r^{\overline{m}_i-A} \phi_i}{\overline{m}_i-\underline{m}_i}\int_0^{r_0} r^{-1-\underline{m}_i}F_i dr\right\rVert^2_{L^2(S)}\\
&\qquad \leq \sum_{\overline{m}_i\geq A} (\overline{m}_i-A)^{2p} \left(\frac{1}{2}\right)^{\overline{m}_i-A} \left\lvert\frac{ \lambda_i   r_0^{\underline{m}_i-A}  }{\overline{m}_i-\underline{m}_i}\int_0^{r_0} r^{-1-\underline{m}_i}F_i dr\right\rvert^2,
\end{split}
\end{align}
where $\left\lvert (\overline{m}_i-A)^{2p}\left(\frac{1}{2}\right)^{\overline{m}_i-A}\right\rvert\leq C(p)$. The rest is already estimated.

We apply the integration by parts to derive,
\begin{itemize}
\item
 for $H_2$,
\begin{align*}
rD_r  \left( r^{\underline{m}_i}\int_0^{r} r^{-1-\underline{m}_i}F_i dr \right)&=\underline{m}_i r^{\underline{m}_i}\int_0^{r} r^{-1-\underline{m}_i}F_i dr + F_i \\
&= r^{\underline{m}_i}  \left(\underline{m}_i \int_0^{r} r^{-1-\underline{m}_i}F_i dr+ \int_0^{r} D_r \left( r^{-\underline{m}_i}F_i \right)dr\right)\\
&= r^{\underline{m}_i}   \int_0^{r} \left( r^{-1-\underline{m}_i}( rD_r) F_i \right)dr.
\end{align*}
\item for $H_3$ in the case $\overline{m}_i<a=A$, similarly,
\begin{align*}
&rD_r  \left( r^{\overline{m}_i}\int_0^{r} r^{-1-\overline{m}_i}F_i dr \right)= r^{\overline{m}_i}   \int_0^{r} \left( r^{-1-\overline{m}_i}( rD_r) F_i \right)dr,
\end{align*}
\item for $H_3$, in the case $\overline{m}_i\geq A$,
\begin{align*}
&rD_r  \left( r^{\overline{m}_i}\int_r^{r_0} r^{-1-\overline{m}_i}F_i dr \right)\\
 &\qquad= r^{\overline{m}_i}  \int_r^{r_0} \left( r^{-1-\overline{m}_i}( rD_r) F_i \right)dr-\left(\frac{r}{r_0}\right)^{\overline{m}_i}F_i(x^\prime_{n-k}, r_0).
\end{align*}
The last term with $\left(\frac{r}{r_0}\right)^{\overline{m}_i}$ factor can be dealt with as \eqref{eq-rDr-p}.
\end{itemize}

As $rD_r F$  has estimates by \eqref{eq-w-pqlm}, we can work as the  $p=0$ case to derive \eqref{eq-Hlm-mp} for $p=1$. For general $p$, we can keep applying these identities to transfer the $(rD_r)^p$ derivatives to $F$, and apply \eqref{eq-w-pqlm} to derive the lemma.
\end{proof}

If $H_{l,m}$'s in Lemma \ref{lem-IntF} satisfy \eqref{eq-l1l2}, we can prove that they have epansions of order $d_S^b$ for any $b \in \mathbb{N}$.
\begin{lemma}\label{lem-wExp}
Assume that $w$ is a function defined in $T_\frac{M}{2}$, that for any fixed $|x^\prime_{n-k}|< \frac{M}{2}, r\in  (0, \frac{M}{4})$,   and any $l, m\in \mathbb{N}$, 
\begin{align*}
D_{x_{n-k}^\prime}^l D^m_r  w,\, D_{x_{n-k}^\prime}^l D^m_r L_S w
\end{align*}
exist, and have estimates
\begin{align}\begin{split}\label{eq-rm-lsw}
|| r^l D^l_r  D_{x_{n-k}^\prime}^m w||_{L^2(S)}&\leq C(T,M, l,m),\\
||r^l D^l_r  D_{x_{n-k}^\prime}^m  L_S w||_{L^2(S)}&\leq C(T,M, l,m),
\end{split}
\end{align}
where the constants $C(T,M, l,m)$'s are independent of $x^\prime_{n-k}, r$.
 In addition, we assume that $w$ satisfies in $T_\frac{M}{2}\cap \{r<\frac{M}{4}\}$, for some $l_1, l_2\in \mathbb{R}$,
\begin{align*}
 r^2 w_{rr} + l_1 rw_r+ l_2 w+ L_S w= F,
\end{align*}
where for any $b\in \mathbb{N}$, $d_S^{-\tau} F$ has a boundary expansion of order $d_S^b$,  for some $\tau$ satisfying $\tau- \frac{n}{2}\in \mathbb{N}$. Then
$d_S^{-\frac{n+2}{2}} w$ has a boundary expansion of order $O(d_S^b)$ for any $b\in \mathbb{N}$ in $T_\frac{M}{2}\cap \{r<\frac{M}{8}\}$.
\end{lemma}

\begin{proof}
First we want to derive the $L^\infty\left(T_\frac{M}{2}\cap \{r<\frac{M}{4}\}\right)$ estimate of $r^l D^l_r  D_{x_{n-k}^\prime}^m  w$, for any $l,m\in \mathbb{N}$. To this end, we set $N$ to be the smallest integer that is greater than $\frac{n}{2}$, and show  $W^{N, 2}$ estimates of $ r^l  D^l_r  D_{x_{n-k}^\prime}^m  w$.

By \eqref{eq-rm-lsw}, for fixed $x_{n-k}^\prime$ and $r$, $L_S w\in L^2(S)$. Hence there is an $f\in L^2(S)$, such that
\begin{align}\label{eq-w-f}
d_S^2\Delta_S w -\frac{n(n+2)}{4}d_S^2 u_S^{\frac{4}{n-2}}w =d_S^2 f.
\end{align}
Notice that $d_S^2 u_S^{\frac{4}{n-2}}$ is uniformly bounded in $S$.
 Denote the metric $G= d_S^{-2} g_S$. Under the rescaled coordinate $\bar{\theta}= d_S^{-1}(P)\theta$ around a point $P\in T_\frac{M}{2}$, the equation \eqref{eq-w-f} is uniformly elliptic. In the metric ball $B_G(P, \frac{1}{2}) \subseteq S$, we apply the interior estimates with respect to $\bar{\theta}$ to derive
\begin{align}\begin{split}\label{eq-w-w22}
||w||_{W^{2,2}_G(B_G(P, \frac{1}{4}))} &\leq C_1 
\left(||w||_{L^{2}_G(B_G(P, \frac{1}{2}))}+d_S^2 ||f||_{L^{2}_G(B_G(P, \frac{1}{2}))}\right)\\
&\leq 2 C_1 C(T,M, 0,0),
\end{split}
\end{align}
 by \eqref{eq-rm-lsw}, where $W^{2,2}_G$ denotes the $W^{2, 2}$ norm under the coordinate $\bar{\theta}$. Here
 $C_1$ is independent of the choice of $x^\prime_{n-k}, r$. Similarly, for fixed $x_{n-k}^\prime, r,$
 \begin{align}\label{eq-wlm-w22}
\left\|(rD_r)^l D_{x_{n-k}^\prime}^m   w\right\|_{W^{2,2}_G(B_G(P, \frac{1}{4}))} \leq C_1 C(T,M, l, m).
\end{align}
Hence $w$ is a local $W^{2,2}$ solution to \eqref{eq-l1l2}.  Furthermore, $w$ is smooth in $T_\frac{M}{2}$ by interior estimates of  \eqref{eq-l1l2}. 

For each $l,m\geq 0$, denote
\begin{align*}
w_{l,m}= (rD_r)^l D_{x_{n-k}^\prime}^m  w,
\end{align*}
which has uniform $W^{2,2}_G(B_G(P, \frac{1}{4}))$ estimate by \eqref{eq-wlm-w22}. Here  ``uniform''  means that the estimate is independent of the choice of $P\in T_\frac{M}{2}\cap \{r<\frac{M}{4}\}$, but still depends on $l,m$.

Write \eqref{eq-l1l2} as
\begin{align}\begin{split}\label{eq-ds2-l1l2}
 d_S(P)^2 \Delta_S w &=  d_S(P)^2 u_S^{\frac{4}{n-2}} w +d_S(P)^2 F\\
 &\qquad - d_S(P)^2 ( r^2 w_{rr} + l_1 rw_r+ l_2 w),
 \end{split}
\end{align}
which is uniformly elliptic with respect to the $\bar{\theta}$ coordinates in $B_G(P, \frac{1}{2}))$. Notice $rw_r= w_{1,0}$ and $r^2w_{rr}= w_{2,0}-w_{1,0}$.

We take $(rD_r)^l D_{x_{n-k}^\prime}^m  $ of \eqref{eq-ds2-l1l2} to derive an elliptic equation of $w_{l, m}$, for which the right hand side is linear in $w_{p, 	q}$'s for $p,q$ satisfying $0\leq p\leq l+2, 0\leq q \leq m$. Then we can apply \eqref{eq-wlm-w22}, and interior estimates to show that
\begin{align}\begin{split}\label{eq-wlm-W42}
&||w_{l,m}||_{W^{4,2}_G(B_G(P, \frac{1}{8}))}
\\ \qquad&\leq C\left( \sum_{\substack{ 0\leq p\leq l+2\\
0\leq q\leq m}} C(T,M, p,q)+ \left\| (rD_r)^l D_{x_{n-k}^\prime}^m  F\right\|_{W^{2,2}_G(B_G(P, \frac{1}{4}))}\right).
\end{split}
\end{align}

Then take $D^2_{\bar{\theta}} (rD_r)^l D_{x_{n-k}^\prime}^m$ of \eqref{eq-ds2-l1l2} to derive an elliptic equation of $D_{\bar{\theta}\bar{\theta}} w_{l, m}$. Applying \eqref{eq-wlm-W42}, we get $W^{6,2}_G(B_G(P, \frac{1}{16}))$ estimate of $w-{l,m}$.

Iterate this step until we get $W^{N, 2}_G$ estimates of $w_{l,m},$ which are independent of the choice of $P$.
Hence we derive $||w_{l,m}||_{L^\infty(T_\frac{M}{2} \cap \{r<\frac{M}{4}\})}\leq C(T, M, l, m)$.

To show the boundary expansion with respect to $d_S$, we do the rescaling $t=\frac{r}{r(P)}$ as in Lemma \ref{lem-v-r-ds}. Then the derivatives of $w$ with respect to $t, x^\prime_{n-k}$ are bounded and independent of the choice of $P$. We can apply the maximum principle to show $w=O(d_S^\frac{n+2}{2})$ as in Lemma \ref{lem-v-r-ds} , and continue to show that $w$ has an expansion of order $O(d_S^\frac{n+2}{2})$ in $T_\frac{M}{2} \cap \{r<\frac{M}{8}\}$ as in Theorem \ref{thm-VExp}.  
\end{proof}

Now it's ready to prove Theorem \ref{thm-Main}.

\begin{proof}[Proof of Theorem \ref{thm-Main}]
Recall that $v= u-u_T$ satisfies \eqref{eq-polar}, and
by Theorem \ref{thm-VExp}, for any $b \in \mathbb{N}$, $d_S^{-\frac{n+2}{2}}v$ has an expansion of order $d_S^b$ in $T_\frac{M}{2}$. As $n\geq 3$, we see that
 $v\in X=\left\lbrace v\in H^1(S): \int_\Omega \frac{v^2(x)}{d_S^2(x)}dx <\infty \right\rbrace $.

 By Appendix \ref{sec-Spec}, for any fixed $x_{n-k}^\prime, r$ such that $|x_{n-k}^\prime|<\frac{M}{2}, 0<r<\frac{M}{2}$,
\begin{align*}
v= \sum_i^\infty A_i(x_{n-k}^\prime, r)\phi_i(\theta),
\end{align*}
where
\begin{align*}
A_i= \int_{S} v(x_{n-k}^\prime, r, \theta) \phi_i(\theta) d\theta.
\end{align*}
Plug into the main equation \eqref{eq-polar} and derive
\begin{align}\label{eq-ODE}
r^2 A_i^{\prime\prime}+(k-1) rA_i^\prime -\lambda_i A_i=  \tilde{F}_i
\end{align}
where
\begin{align*}
\tilde{F}_i=\int_{S} \left(r^{\frac{n-2}{2}} u_S^{\frac{6-n}{n-2}} v^2 \cdot F (r^{\frac{n-2}{2}} u_S^{-1}v) - r^2\Delta_{\mathbb{R}^{n-k}}v \right)\phi_i d\theta.
\end{align*}
The ODE \eqref{eq-ODE} has homogeneous solutions $r^{\underline{m}_i}, r^{\overline{m}_i}$, and the general solution is
\begin{align}\begin{split}\label{eq-ODEGenSol}
A_i&=C_1 r^{\underline{m}_i} +C_2 r^{\overline{m}_i}
+\frac{r^{\underline{m}_i}}{\overline{m}_i-\underline{m}_i}\int_r^{r_0} r^{-1-\underline{m}_i}\tilde{F}_i dr \\
&\qquad - \frac{r^{\overline{m}_i}}{\overline{m}_i-\underline{m}_i}\int_r^{r_0} r^{-1-\overline{m}_i}\tilde{F}_i dr. 
\end{split}
\end{align}
To solve out $C_1, C_2$, first we take $r=r_0$ for fixed $r_0=\frac{M}{2}$, where
\begin{align*}
A_i(x_{n-k}^\prime, r_0)= C_1 r_0^{\underline{m}_i} +C_2 r_0^{\overline{m}_i}.
\end{align*}
Secondly, since $|v|\leq C$, multiply \eqref{eq-ODEGenSol} by $r^{-\underline{m}_i}$, and let $r\rightarrow 0$, 
\begin{align*}
0=C_1 +\frac{1}{\overline{m}_i-\underline{m}_i}\int_0^{r_0} r^{-1-\underline{m}_i}\tilde{F}_i dr.
\end{align*}
So we have
\begin{align}\begin{split}\label{eq-Ai}
A_i&=   \left(A_i(x_{n-k}^\prime, r_0) r_0^{-\overline{m}_i} - \frac{r_0^{\underline{m}_i-\overline{m}_i}}{\overline{m}_i-\underline{m}_i}\int_0^{r_0} r^{-1-\underline{m}_i}\tilde{F}_i dr  \right) r^{\overline{m}_i}\\
&\qquad
-\frac{r^{\underline{m}_i}}{\overline{m}_i-\underline{m}_i}\int_0^{r} r^{-1-\underline{m}_i}\tilde{F}_i dr 
- \frac{r^{\overline{m}_i}}{\overline{m}_i-\underline{m}_i}\int_r^{r_0} r^{-1-\overline{m}_i}\tilde{F}_i dr,
\end{split} 
\end{align}
and
\begin{align}\begin{split}\label{eq-ODESoln}
v =\sum_{i=1}^\infty A_i \phi_i&= \sum_{i=1}^\infty \frac{v_i(x_{n-k}^\prime, r_0) r^{\overline{m}_i} \phi_i}{r_0^{\overline{m}_i}} - \sum_{i=1}^\infty \frac{r_0^{\underline{m}_i-\overline{m}_i}  r^{\overline{m}_i} \phi_i}{\overline{m}_i-\underline{m}_i}\int_0^{r_0} r^{-1-\underline{m}_i}\tilde{F}_i dr \\
&\qquad
-\sum_{i=1}^\infty \frac{r^{\underline{m}_i} \phi_i}{\overline{m}_i-\underline{m}_i}\int_0^{r} r^{-1-\underline{m}_i}\tilde{F}_i dr 
-\sum_{i=1}^\infty  \frac{r^{\overline{m}_i} \phi_i}{\overline{m}_i-\underline{m}_i}\int_r^{r_0} r^{-1-\overline{m}_i}\tilde{F}_i dr. 
\end{split}
\end{align}
Take an $\epsilon>0$, and less than $i_0:=\min\{\overline{m}_1, \frac{n-2}{2}, 2\}.$  
Applying Lemma \ref{lem-vi-1} with $A=\epsilon$,
the term 
\begin{align}\label{eq-sum-vi-a}
\sum_{i=1}^\infty \frac{v_i(x_{n-k}^\prime, r_0) r^{\overline{m}_i} \phi_i}{r_0^{\overline{m}_i}} 
\end{align}
can be written as the expansion \eqref{eq-HExp-1} with
\eqref{eq-Hlm-mp-1} holds in $T_\frac{M}{2} \cap \{r<\frac{M}{4}\}$. As $\epsilon< \overline{m}_1$, we only have one term $H_{A, 0}(x_{n-k}^\prime,  r, \theta) r^{\epsilon}$  in the expansion.

By Theorem \ref{thm-VExp}, $d_S^{- \frac{n+2}{2}} v$ has an expansion of order $d_S^b$ for any $b\geq n$ in $T_\frac{M}{2}$. So does $d_S^\frac{n-2}{2} u_S$. Then by Lemma \ref{lem-Fw1wk},
\begin{align*}
 \tilde{F}&:=r^{\frac{n-2}{2}} u_S^{\frac{6-n}{n-2}} v^2 \cdot F (r^{\frac{n-2}{2}} u_S^{-1}v) - r^2\Delta_{\mathbb{R}^{n-k}}v\\
 &= r^{\frac{n-2}{2}}  d_S^{\frac{n-2}{2}+n}(d_S^\frac{n-2}{2}u_S)^{p-2}(d_S^{-\frac{n+2}{2}}v)^2 \cdot F(r^\frac{n-2}{2} (d_S^\frac{n-2}{2}u_S)^{-1} d_S^\frac{n-2}{2}v) - r^2\Delta_{\mathbb{R}^{n-k}}v,
\end{align*}
equals $r^{\min\{\frac{n-2}{2},2\}}  d_S^{\frac{n+2}{2}}$ times a function which has an expansion of order $d_S^b$ for any $b\geq n$. 

Applying Lemma \ref{lem-IntF} with $a=A=\epsilon,$   $j=0$ and  $F(x^\prime_{n-k}, r, \theta) = \tilde{F}$,
we have that the last three terms in
\eqref{eq-ODESoln} can be written as  
$H_{A, 0}(x_{n-k}^\prime,  r, \theta) r^{\epsilon}$ with the estimate \eqref{eq-Hlm-mp} holds in $T_\frac{M}{2} \cap \{r<\frac{M}{4}\}$.
In sum,
\begin{align*}
v(x^\prime_{n-k}, r, \theta) = r^{\epsilon}w_0(x^\prime_{n-k}, r, \theta),
\end{align*}
 in $T_\frac{M}{2} \cap \{r<\frac{M}{4}\}$, where $w_0$ satisfies that for any $p, q\in \mathbb{N}$,
\begin{align*}
|| r^p D^p_r  D_{x_{n-k}^\prime}^q  w_0||_{L^2(S)}&\leq C(T, M, p,q),\\
||r^p  D^p_r D_{x_{n-k}^\prime}^q  L_S w_0||_{L^2(S)}&\leq C(T, M, p,q),
\end{align*}
where $C(T, M, p,q)$ is independent of $x^\prime_{n-k}, r$. Then by \eqref{eq-polar}, $w_0$ is a local $W^{2,2}$ solution of
\begin{align*}
 N_0 w_0 + L_Sw_0= r^{\frac{n-2}{2}-\epsilon}u_S^{\frac{6-n}{n-2}} v^2 F(r^\frac{n-2}{2}u_S^{-1} v) -r^{2-\epsilon}\Delta_{\mathbb{R}^{n-k}}v
\end{align*}
where
\begin{align*}
N_{0}=r^2 D_{rr} +(2\epsilon+k-1) r D_r+ \epsilon(k+\epsilon-2).
\end{align*}
By Lemma \ref{lem-wExp}, $d_S^{-\frac{n+2}{2}} w_0$ has boundary expansion of order $d_S^b$ for any      $b\in \mathbb{N}$ in $T_\frac{M}{2} \cap \{r<\frac{M}{8}\}$. In other words, $v$ has an expansion of order $r^0$ with  $O(d_S^\frac{n+2}{2})$ coefficients  in $T_\frac{M}{2} \cap \{r<\frac{M}{8}\}$.  By Lemma \ref{lem-M4-M2}, we can improve the domain $T_\frac{M}{2} \cap \{r<\frac{M}{8}\}$ to $T_\frac{M}{2}$.

Inductively, we assume that $v$ has an expansion of order $r^a$ with $O(d_S^\frac{n+2}{2})$ coefficients for some $a\in I$ in $T_\frac{M}{2}$, i.e., there are functions
$\tilde{c}_{i, j}$'s, $\tilde{R}_a$ defined on $T_\frac{M}{2}$, and an $\epsilon>0$, such that,
\begin{align}\label{eq-Exp-r1}
v= \sum_{i\in J, i\leq a}\sum_{j=0}^{\tilde{N}_i} \tilde{c}_{i,j} r^i(\log r)^j +\tilde{R}_a,
\end{align}
 where $d_S^{-\frac{n+2}{2}} \tilde{ c}_{i,j}$ and $d_S^{-\frac{n+2}{2}}  \cdot r^{-a-\epsilon}  \tilde{R}_a$ have boundary expansions up to order $d_S^b$ for any integer $b\in \mathbb{N}$. In addition, $\tilde{c}_{i, j}$'s are independent of $r$. 
 
 We prove that $v$ has an expansion of order $r^{a^+}$ with $O(d_S^\frac{n+2}{2})$ coefficients in $T_\frac{M}{2}$, where we recall that $a^+$ is smallest element in $I$ that is larger than $a.$

First we adjust $\epsilon$ if necessary, such that 
\begin{align}\label{eq-epsi}
\overline{m}_i>a_+\text{ implies that } \overline{m}_i>a_++\epsilon.
\end{align}

Again by Lemma \ref{lem-vi-1} with $A=a^+ +\epsilon$, the term \eqref{eq-sum-vi-a}
can be written as the expansion \eqref{eq-HExp-1} with
\eqref{eq-Hlm-mp-1} holds in $T_\frac{M}{2} \cap \{r<\frac{M}{4}\}$.

To deal with rest terms in \eqref{eq-ODESoln}, we have the following lemma, 
\begin{lemma}\label{lem-FTild}
If $v$ has an expansion of order $r^a$ with $O(d_S^\frac{n+2}{2})$ coefficients in $T_\frac{M}{2}$, then $\tilde{F}=r^{\frac{n-2}{2}} u_S^{\frac{6-n}{n-2}} v^2 \cdot F (r^{\frac{n-2}{2}}  u_S^{-1}v) - r^2\Delta_{\mathbb{R}^{n-k}}v$ has an expansion of order $r^{a^+}$ with $O(d_S^{\frac{n+2}{2}})$ coefficients  in $T_\frac{M}{2}$.
\end{lemma}
\begin{proof}
The proof is by formal computation.  We claim that $v^2$ has  an expansion of order $r^{a^+}$ with $O(d_S^{n+2})$ coefficients. In fact, if $a=0$, $v^2= r^{2i_0} w_0^{2}$ has an expansion of order $r^{i_0}$ with $O(d_S^\frac{n+2}{2})$ coefficients. Here $i_0= 0^+.$ If $a> 0$, by \eqref{eq-Exp-r1}, and the fact that the leading term of $v$ is $r^{i_0}$ ($r^0$ term has zero coefficient), $v^2$ has an expansion of order $r^{k+i_0}$ with $O(d_S^{n+2})$ coefficients. As $a+i_0\geq k^+,$  $v^2$ has an expansion of order $r^{a^+}$ with $O(d_S^{n+2})$ coefficients.

By \eqref{eq-USP3}, $u_S^{-1} v$ has an expansion with $O(d_S^{n})$ coefficients. Then by \eqref{eq-USP2},
$r^{\frac{n-2}{2}} u_S^{\frac{6-n}{n-2}} v^2 \cdot F (r^{\frac{n-2}{2}}  u_S^{-1}v) )$ has an expansion of order $r^{a^+}$ with $O(d_S^{\frac{3n}{2}-1})$ coefficients.

By the assumption, $r^2\Delta_{\mathbb{R}^{n-k}}v$ has an expansion of order $r^{2+a}$ with $O(d_S^\frac{n+2}{2})$ coefficients, where $2+a\geq a^+$, which concludes the lemma.
\end{proof}
Let the expansion of $\tilde{F}$ be
\begin{align}\label{eq-exp-Ftilde}
\tilde{F}= \sum_{l\in I, l\leq a^+} \sum_{m=0}^{\tilde{N}_l} \tilde{F}_{l, m}(x^\prime_{n-k}, \theta) r^l (\log r)^m +\tilde{R}_{\tilde{F}, a^+}(x^\prime_{n-k}, r, \theta).
\end{align}
Applying Lemma \ref{lem-IntF} with $A= a^++\epsilon$,
to each term $\tilde{F}_{l, m}(\theta) r^l (\log r)^m$, and $\tilde{R}_{\tilde{F}, a^+}(r, \theta)$ in \eqref{eq-exp-Ftilde}, then summing up with the expansion of \eqref{eq-sum-vi-a} of form \eqref{eq-HExp-1},
we verify that in  in $T_\frac{M}{2}\cap \{r<\frac{M}{4}\}$, $v$ has an expansion of form
\begin{align}\label{eq-VExp11}
v= \sum_{l\in I, l\leq a^+} \sum_{m=0}^{\tilde{N}_l} \tilde{c}_{l, m}(x^\prime_{n-k},  \theta) r^l (\log r)^m +\tilde{R}_{a^+}(x^\prime_{n-k},  r, \theta),
\end{align}
where $\tilde{c}_{l, m}$, $r^{-a^+-\epsilon} \tilde{R}_{a^+}$ satisfy the estimates \eqref{eq-rm-lsw}. If the expansion \eqref{eq-HExp1} for $\tilde{R}_{\tilde{F}, a^+}(r, \theta)$ has  a term $H_{A, 1}r^A \log r$ with nonzero $H_{A, 1}$, we have to adjust $\epsilon$ smaller. Actually we already did it in \eqref{eq-epsi}, which guarantees $H_{A,1}=0$.

Plugging \eqref{eq-VExp11} into \eqref{eq-polar}, by assuming $\tilde{R}_{a^+}= r^{a^++\epsilon}w_{a^+}$, we derive the following equations for $\tilde{c}_{l,m}$ and $w_{a^+}$.
\begin{align*}
&l(k+l-2)\tilde{c}_{l,m}+ L_S\tilde{c}_{l,m}\\
&\qquad= \tilde{F}_{l,m}-(m+1)(2l+k-2)\tilde{c}_{l, m+1}- (m+2)(m+1)\tilde{c}_{l, m+2},
\end{align*}
where $\tilde{c}_{l, m+1}=0$, if $m\geq \tilde{N}_l$. And,
\begin{align*}
&r^2(w_{a^{+}})_{rr} +(2a^{+} +2\epsilon+k-1)r(w_{a^{+}})_r\\
&\qquad+(a^{+} +\epsilon)(a^{+} +\epsilon+k-2)w_{a^{+}}+L_S w_{a^{+}}=\tilde{R}_{\tilde{F}, a^{+}}r^{-a^{+} -\epsilon}.
\end{align*}

For every $l$, $\tilde{c}_{l, \tilde{N}_l}$ satisfies
\begin{align*}
l(k+l-2)\tilde{c}_{l,\tilde{N}_l}+ L_S\tilde{c}_{l, \tilde{N}_l}= \tilde{F}_{l,\tilde{N}_l}.
\end{align*}
By \cite{JiangXiao} or \cite{HanJiang},  $\tilde{c}_{l, \tilde{N}_l}$ has an expansion of order $d_S^\frac{n+2}{2}$. Then we inductively prove that $c_{l, m}$ has an expansion of order $d_S^\frac{n+2}{2}$ for any $m$. Theorem \ref{lem-wExp} implies that $w_{a^+}$ has an expansion of order $d_S^\frac{n+2}{2}$  in $T_\frac{M}{2}\cap \{r<\frac{M}{8}\}$. This verifies that $v$ has a boundary expansion of order $r^{a^+}$ with $O(d_S^\frac{n+2}{2})$ coefficients in $T_\frac{M}{2}\cap \{r<\frac{M}{8}\}$. Finally, we apply Lemma \ref{lem-M4-M2} to conclude the theorem.
\end{proof}

\bigskip
\appendix

\section{Spectral theorem for singular elliptic operators}\label{sec-Spec}
Let $S$ be a Lipschitz  $n$ dimensional   Riemannian manifold with codimension one  boundary.  Assume that $d(x)$ is a Lipschitz  defining function of $\partial S$, which means that $d(x)=0$ if and only if $x\in \partial S$, $ |\nabla d|\leq C$ in $S$ for some $C$, and for any $x\in S$,
\begin{align}\label{eq-def-fun}
C^{-1} \leq \frac{d(x)}{dist(x, \partial S)}\leq C.
\end{align}

Consder the eigenvalue problem for
\begin{align*}
L[u] = -\Delta u+\frac{\kappa}{d^2(x)}u =\lambda u,
\end{align*}
where $\Delta$ is the Laplace-Beltrami operator on $S$ and $\kappa>0$ is a constant. Let $X=\left\lbrace u\in H^1(S): \int_S \frac{u^2(x)}{d^2(x)}dvol <\infty \right\rbrace$, with the norm
\begin{align*}
||u||_X := \left\lbrace \int_S \left(|\nabla u|^2 +\frac{u^2(x)}{d^2(x)}\right) dvol \right\rbrace^\frac{1}{2},
\end{align*}
which is also denoted as $||u||_d.$

Claim 1: $X = H^1_0 (S)$.
\begin{proof}
First we show that
$X \subseteq H^1_0 (S)$.
Let ${\eta}_{\epsilon}(r)$ be a cutoff function, which is $0$ when $r \leq \epsilon$ and $1$ when $r\geq 2\epsilon$. In addition, $|\eta^\prime_\epsilon(r)| \leq \frac{2}{\epsilon}$. For any $u\in X, u(x)\eta_\epsilon(d(x)) \in H^1_0(S)$,
\begin{align*}
&\int_S |\nabla (u(x)- u(x) \eta_\epsilon(d(x))|^2 dvol \\
&\qquad \leq  \int_S |1-\eta_\epsilon(d(x))  |^2|\nabla u|^2 dvol
+ \int_S | \nabla   (\eta_\epsilon(d(x)))|^2 u^2(x) dvol\\
&\qquad \leq  \int_S \left( |\nabla u|^2 
+ 4 | \nabla d(x)|^2 \frac{u^2(x)}{d^2(x)}\right)\cdot m\{0\leq d(x) \leq 2\epsilon\}  dvol  
\end{align*}
which tends to $0$ as $\epsilon \rightarrow 0$.

Secondly we show that $H^1_0 (S) \subseteq  X$.
By Hardy's inequality and \eqref{eq-def-fun}, for any
$u \in H^1_0(S)$,
\begin{align*}
\left\|\frac{u}{d}\right\|_{L^2(S)} \leq  C(S, l)\|\nabla u\|_{L^2(S)}.
\end{align*}
Thus $\|u\|_{X} \leq C\|u\|_{H^1_0}$  and the claim is verifed.
\end{proof}

Claim 2: For any $f \in L^2(S),$ there is a unique $u= K[f] \in X$ such that $L(u)=f$ in the integral sense, i.e. for any $\varphi \in X$, 
\begin{align*}
\int_S \left( \nabla u \cdot \nabla \varphi + \frac{\kappa}{d^2(x)} u(x) \varphi (x) \right) dvol = \int_S f(x) \varphi(x) dvol.
\end{align*} 

\begin{proof}
It follows by a standard argument using the Lax-Milgram Theorem. In fact, we define for any $u, v \in X$,
\begin{align*}
B(u, v) = \int_S \left( \nabla u \cdot \nabla v + \frac{\kappa}{d^2(x)} u(x) v (x) \right) dvol.
\end{align*}
Easy to check
\begin{align*}
|B(u, v)| &\leq C(\kappa) ( ||u||_X \cdot ||v||_X ),\\
B(u, u) &\geq c(\kappa) ||u||_X^2, 
\end{align*}
which implies for any $f\in L^2$, there is a unique $u$, such that
\begin{align*}
B(u, \varphi) = \int_S f \varphi dvol,
\end{align*}
for any $\varphi \in X$,
by the Lax-Milgram Theorem.

\end{proof}

\begin{lemma}\label{lem-A1}
 $K$ defined in claim 2 satisfies the following properties,

(a) $||K[f]||_X \leq C||f||_{L^2 (S)}$\\

(b) $(K[f], g)_{L^2(S)} = (f, K[g])_{L^2(S)}$\\

(c) $K: L^2(S) \rightarrow X  \hookrightarrow H^1_0 (S) \hookrightarrow L^2(S)$ is a compact operator.\\

(d)For $u\in X$, $L(u)=\lambda u$ if and only if $\lambda K[u]=u.$
\end{lemma}

\begin{proof}
(a) For any $f$, assume $u=K[f]$. Then
\begin{align*}
c(k)||u||_X^2 \leq B(u, u)= \int_S uLu dvol  =\int_S fu dvol \leq ||u||_{L^2} \cdot ||f||_{L^2},
\end{align*}
which implies (a).

(b) $(K[f], g)_{L^2(S)}= B(K[f], K[g])= (f, K[g])_{L^2(S)}$.

(c) $L^2(S) \rightarrow X \hookrightarrow H^1_0(S)$ is bounded, and $H^1_0(S) \hookrightarrow L^2(S)$ is compact, implying $K$ is compact.

(d) is trivial.
\end{proof}

Based on (a)-(c), $K$ is a self-adjoint compact operator on $L^2(S)$. So it has a complete set of $L^2(S)$-orthonormal eigenfunctions $\{u_j\}_{j=1}^\infty$, which are also eigenfunctions of $L$.

\bigskip
\section{Fundamental lemmas}\label{sec-lem}
In this section, we show some lemmas needed in the proof of Theorem \ref{thm-Main}.

First we introduce the operator $ T$ on $L^2(S)$,  defined as, for $w(\theta) =\sum_i B_i \phi_i(\theta) $,
\begin{align*}
Tw = (-L_S)^\frac{1}{2} w = \sum_i \sqrt{\lambda_i} B_i \phi_i.
\end{align*}
Easy to check $T$ is self-adjoint, and $T\phi_i= \sqrt{\lambda_i}\phi_i$. The following lemma is well known.
\begin{lemma}\label{lem-Tw}
Assume $w\in H^1(S)\cap C^2(S)$, and $w= O(d_S^l)$, $L_S w= O(d_S^{l-2})$  for some $l \geq 1$.  Then $Tw\in L^2(S)$ and
\begin{align*}
(Tw,  Tw)_{L^2(S)}\leq (w, - L_S w)_{L^2(S)}.
\end{align*}
\end{lemma}
\begin{proof}
Assum $w= \sum_i B_i \phi_i(\theta)$. Then
\begin{align*}
(Tw, Tw)_{L^2(S)} &= \sum_i (\sqrt{\lambda_i}B_i)^2 \\
 &= \sum_i ( -B_i L_S \phi_i,  w)_{L^2(S)}\\
&= \sum_i ( B_i \phi_i, - L_S  w)_{L^2(S)}\\
&= ( w, - L_S  w)_{L^2(S)}.
\end{align*}
Here we switch $L_S$ by applying Stoke's Theorem and the fact that $\phi_i= O(d_S^\frac{n+2}{2}), w=O(d_S)$. $L_S w=O(d_S^{-1})\notin L^2(S)$ if $l=1$, but we can regrad $(\cdot, L_S w)_{L^2(S)}$ as a bounded linear operaton on $X$. So $ \sum_i ( B_i \phi_i, - L_S  w)_{L^2(S)}$ converges to $( w, - L_S  w)_{L^2(S)}$.
\end{proof}

Next we show,
\begin{lemma}\label{lem-vi-1}
Fix an index $A>0$.
Assume that in $T_\frac{M}{2}$, $d_S^{-\frac{n+2}{2}} v$ has the boundary expansion of order $d_S^b$ for any $b \in \mathbb{N}$. In addition, we assume in $T_\frac{M}{2}$,
for any $p, q, l, m\in \mathbb{N}$, it holds
\begin{align}\label{eq-v-pqlm}
r^p d_S^m \left\lvert  D_r^p D_{x^\prime_{n-k}}^q  D_{z_S}^l D_{d_S}^m (d_S^{-\frac{n+2}{2}} v) \right\rvert&\leq C(T, M, p, q, l, m).
\end{align}

 Denote
\begin{align*}
v_i(x^\prime_{n-k}, r)=\int_S v(x^\prime_{n-k}, r,\theta)\phi_i d\theta.
\end{align*}
Then for $r_0=\frac{M}{2}$, the summation
\begin{align}\label{eq-sum-vi-1}
\sum_{i=1}^\infty \frac{v_i(x^\prime_{n-k}, r_0) r^{\overline{m}_i} \phi_i}{r_0^{\overline{m}_i}}
\end{align}
has an expansion of form,  
\begin{align}\label{eq-HExp-1}
\sum_{l\in I, l<A}  H_{l, 0}(x_{n-k}^\prime,  \theta) r^l  +H_{A, 0}(x_{n-k}^\prime,  r, \theta) r^A,
\end{align}
where in $T_\frac{M}{2}\cap \{r<\frac{M}{4}\}$, all coefficients $H_{l, m}$'s  satisfy, that for any fixed $x_{n-k}^\prime,  r$,
and for any $p, q\in \mathbb{N}$,
\begin{align}\begin{split}\label{eq-Hlm-mp-1}
\|r^p D^p_r D_{x_{n-k}^\prime}^q  (H_{l, m})\|_{L^2(S)}&\leq C(T, M, l,m,p, q) ,\\
\|r^p D^p_r D_{x_{n-k}^\prime}^q L_S (H_{l, m}) \|_{L^2(S)}&\leq C(T, M, l,m,p, q).\end{split}
\end{align}
\end{lemma}
\begin{proof}
First \eqref{eq-sum-vi-1} can be written as
\begin{align}\label{eq-sum-vi}
\sum_{\overline{m}_i<A} \frac{v_i(x^\prime_{n-k}, r_0)  \phi_i}{r_0^{\overline{m}_i}}r^{\overline{m}_i}+\left( \sum_{\overline{m}_i\geq A} \frac{v_i(x^\prime_{n-k}, r_0) r^{\overline{m}_i-A} \phi_i}{r_0^{\overline{m}_i}}\right) r^{A},
\end{align}
which is in the form of \eqref{eq-HExp-1}. 

Then we show \eqref{eq-Hlm-mp-1} when $p=q=0.$
As $r< r_0$, we derive that
\begin{align*}
&\left\lVert L_S  \sum_{\overline{m}_i\geq A} \frac{v_i(x^\prime_{n-k}, r_0) r^{\overline{m}_i-A} \phi_i}{r_0^{\overline{m}_i}} \right\rVert^2_{L^2(S)}\\
&\qquad=\sum_{\overline{m}_i\geq A} \left(  \frac{v_i(x^\prime_{n-k}, r_0) r^{\overline{m}_i-A} \lambda_i}{r_0^{\overline{m}_i}}\right)^2\\
&\qquad\leq \sum_{\overline{m}_i\geq A} \left(\lambda_i v_i(x^\prime_{n-k}, r_0)  \right)^2 r_0^{-2A}\\
&\qquad\leq C(T, M, A)||L_S v(x^\prime_{n-k}, r_0, \cdot)||_{L^2(S)}^2,
\end{align*}
where $L_S v$ is bounded by \eqref{eq-v-pqlm}.

As in \eqref{eq-sum-vi}, only $v_i$ depends on $x_{n-k}^\prime$, so if $p=0$, applying \eqref{eq-v-pqlm},  \eqref{eq-Hlm-mp-1} can be derived in the same way.

If $p\neq 0$, the only coefficient that depends on $r$ is $H_{A, 0}$. If $r\leq \frac{1}{2}r_0$, $(rD_r)^p H_{A, 0}$ can be estimated just as \eqref{eq-rDr-p}. 
\end{proof}

\begin{lemma}\label{lem-M4-M2}
If for some $a\in \mathbb{R}^+$, $v$ has a boundary expansion of order $r^a$ with $O(d_S^\frac{n+2}{2})$ coefficients in $T_\frac{M}{2}\cap \{r<\frac{M}{8}\}$,  and in addition, $d_S^{-\frac{n+2}{2}} v$ has a boundary expansion of order $d_S^b$ for any $b\in \mathbb{N}$ in $T_\frac{M}{2}$, then
$v$ satisfies a boundary expansion of order $r^a$ with $O(d_S^\frac{n+2}{2})$ coefficients in $T_\frac{M}{2}$.
\end{lemma}
\begin{proof}
By the assumption, in $T_\frac{M}{2}\cap \{r<\frac{M}{8}\},$
\begin{align}\label{eq-v-M4}
v= \sum_{i\in J, i\leq a}\sum_{j=0}^{\tilde{N}_i} \tilde{c}_{i,j} r^i(\log r)^j +\tilde{R}_a,
\end{align}
where $d_S^{-\frac{n+2}{2}} \tilde{ c}_{i,j}$ and $d_S^{-\frac{n+2}{2}}  \cdot r^{-a-\epsilon}  \tilde{R}_a$ have boundary expansions up to order $d_S^b$ for any integer $b\in \mathbb{N}$. In addition, $\tilde{c}_{i, j}$'s are independent of $r$. 
Then in $T_{\frac{M}{2}}$, we express $v$ as
\begin{align*}
v= \sum_{i\in J, i\leq a}\sum_{j=0}^{\tilde{N}_i} \tilde{c}_{i,j} r^i(\log r)^j +\left(v- \sum_{i\in J, i\leq a}\sum_{j=0}^{\tilde{N}_i} \tilde{c}_{i,j} r^i(\log r)^j \right).
\end{align*}
For $r<\frac{M}{8}$, it is the same as \eqref{eq-v-M4}.
 If $r\geq \frac{M}{8}$, by the assumption,
\begin{align*}
d_S^{-\frac{n+2}{2}}r^{-a-\epsilon}\left(v- \sum_{i\in J, i\leq a}\sum_{j=0}^{\tilde{N}_i} \tilde{c}_{i,j} r^i(\log r)^j \right)
\end{align*}
has a boundary expansion up to order $d_S^b$ for any integer $b\in \mathbb{N}$, which confirms with Definition \ref{def-exp-r}.
\end{proof}

\end{document}